\documentclass[12pt, reqno]{amsart}
\usepackage[table]{xcolor}
\usepackage{amsmath,amssymb, amsfonts,amscd,psfrag,graphicx,stmaryrd}
\usepackage{epsfig}
\usepackage{amsthm}
\usepackage{fullpage}
\usepackage{verbatim}
\usepackage{listings}
\usepackage[table]{xcolor} 
\usepackage{mathtools}
\DeclarePairedDelimiter{\ceil}{\lceil}{\rceil}
\DeclarePairedDelimiter\floor{\lfloor}{\rfloor}
\newcommand{\lebn}

\usepackage[misc]{ifsym}
\usepackage{caption}
\usepackage{float}
\usepackage{amsmath}
\usepackage{subfigure}

\theoremstyle{plain}
\newtheorem{proposition}[equation]{Proposition}

\newtheorem{theorem}[equation]{Theorem}
\newtheorem{conjecture}[equation]{Conjecture}

\newtheorem{corollary}[equation]{Corollary}
\newtheorem{lemma}[equation]{Lemma}

\theoremstyle{definition}

\newtheorem{remark}[equation]{Remark}

\numberwithin{equation}{section}

\newcommand{\D}{\Delta}

\usepackage[a4paper,left=1.5cm,right=1.5cm,top=2cm,bottom=2cm]{geometry}

\allowdisplaybreaks

\usepackage{tikz}
\usepackage{verbatim}
\usetikzlibrary{shapes,shadows,calc}
\usepgflibrary{arrows}
\usetikzlibrary{arrows, decorations.markings, calc, fadings, decorations.pathreplacing, patterns, decorations.pathmorphing, positioning}
\tikzset{nodc/.style={circle,draw=blue!50,fill=pink!80,inner sep=1.6pt}}
\tikzset{nodr/.style={circle,draw=black,fill=green!50!black,inner sep=2.5pt}}
\tikzset{nodel/.style={circle,draw=black,inner sep=2.2pt}}
\tikzset{nodinvisible/.style={circle,draw=white,inner sep=2pt}}
\tikzset{nodpale/.style={circle,draw=gray,fill=gray,inner sep=1.6pt}}

\tikzset{nod1/.style={circle,draw=black,fill=black,inner sep=1pt}}
\tikzset{nod2/.style={circle,draw=black,fill=blue!75!black,inner sep=1.6pt}}
\tikzset{nodgs/.style={circle,draw=black,dotted,fill=gray,inner sep=1.6pt}}
\tikzset{nod3/.style={circle,draw=black,fill=black,inner sep=1.8pt}}
\tikzset{noddiam/.style={diamond,draw=black,inner sep=2pt}}
\tikzset{nodw/.style={circle,draw=black,inner sep=1.8pt}}

\usepackage[colorlinks, urlcolor=blue, linkcolor=blue, citecolor=blue]{hyperref}

\makeatletter
 \def\@textbottom{\vskip \z@ \@plus 10pt}
 \let\@texttop\relax
\makeatother

\usepackage[utf8]{inputenc}
\usepackage[T1]{fontenc}

\begin{document}

\bibliographystyle{plain}

\title[2-distance coloring of planar graphs with girth five]{On $2$-distance $16$-coloring of planar graphs with maximum degree at most five}  


\author{Zakir Deniz}
\thanks{The author is supported by T\" UB\. ITAK, grant no:122F250}
\address{Department of Mathematics, Duzce University, Duzce, 81620, Turkey.}
\email{zakirdeniz@duzce.edu.tr}

\date{\today}
\thanks{}
\subjclass[2010]{}

\begin{abstract}


A vertex coloring of a graph $G$ is called a 2-distance coloring if any two vertices at a distance at most $2$ from each other receive different colors. Suppose that $G$ is a planar graph with a maximum degree at most $5$. We prove that $G$ admits a $2$-distance $16$ coloring, which improves the result given by Hou et al. (Graphs and Combinatorics 39:20, 2023).

\end{abstract}
\maketitle


\section{Introduction}

We consider only finite simple graphs throughout this paper, and refer to \cite{west} for terminology and notation not defined here. Let $G$ be a graph, we use $V(G),E(G),F(G),$ and $\D(G)$ to denote the vertex, edge and face set, and the maximum degree of $G$, respectively. If there is no confusion in the context, we abbreviate $\D(G)$ to $\D$. 
A 2-distance coloring is a proper vertex coloring where two vertices that are adjacent or have a common neighbour receive different colors, and the smallest number of colors for which $G$ admits a 2-distance coloring is known as the 2-distance chromatic number $\chi_2(G)$ of $G$. 
In \cite{daniel}, Cranston has provided a comprehensive survey on 2-distance coloring and related coloring concepts.

In 1977, Wegner \cite{wegner} posed the following conjecture.

\begin{conjecture}\label{conj:main}
For every planar graph $G$, $\chi_2(G) \leq 7$ if $\Delta=3$, $\chi_2(G) \leq \Delta+5$ if $4\leq \Delta\leq 7$, and $\chi_2(G) \leq  \floor[\big]{\frac{3\Delta}{2}}+1$ if $\Delta\geq 8$, 
\end{conjecture}

Thomassen \cite{thomassen} (independently by Hartke et al. \cite{hartke}) proved the conjecture for planar graphs with $\Delta = 3$, by showing that every cubic planar graph is 2-distance colorable with $7$ colors. For general planar graphs,  van den Heuvel and McGuinness \cite{van-den} showed that $\chi_2(G) \leq 2\Delta + 25$, while the bound $\chi_2(G) \leq \ceil[\big]{ \frac{5\D}{3}}+78$ was proved by Molloy and Salavatipour \cite{molloy}.\medskip






The conjecture remains open, even for small $\D$. Zhu \cite{zhu} proved that $\chi_2(G)\leq 13$ for planar graph with $\D =4$. On the other hand,  Zhu and Bu \cite{zhu2} showed that  $\chi_2(G)\leq 20$ for planar graph with maximum degree $\D \leq 5$. This bound was later reduced to $19$ by Chen et al. \cite{chen}, and to 18 by Hou et al. \cite{hou}. Recently, Aoki \cite{aoki} and Zou et al. \cite{zou} (independently) in unpublished papers have proved that $\chi_2(G)\leq 17$.    In this paper, we further improve the current bound by reducing it by one.

\begin{theorem}\label{thm:main}
If $G$ is a planar graph with maximum degree at most $5$, then $\chi_2(G) \leq 16$. 
\end{theorem}

Given a planar  graph $G$, 
we denote by $\ell(f)$ the length of a face $f$ and by $d(v)$ the degree of a vertex $v$. 
A $k$-vertex is a vertex of degree $k$. A $k^{-}$-vertex is a vertex of degree at most $k$ while a $k^{+}$-vertex is a vertex of degree at least $k$. A $k$ ($k^-$ or $k^+$)-face is defined similarly. A vertex $u\in N(v)$ is called $k$-neighbour (resp. $k^-$-neighbour, $k^+$-neighbour) of $v$ if $d(u)=k$ (resp. $d(u)\leq k$, $d(u)\geq k$). 
For a vertex $v\in V(G)$, we use $n_i(v)$ to denote the number of $i$-vertices adjacent to $v$. 
We denote by $d(u,v)$ the distance between $u$ and $v$ for any pair $u,v\in V(G)$. Also, we set $N_i(v)=\{u\in V(G) \ | \  1 \leq d(u,v) \leq i \}$ for $i\geq 1$, so $N_1(v)=N(v)$ and let $d_2(v);=|N_2(v)|$.

For $v \in V(G)$, we use $m_k(v)$ to denote the number of $k$-faces incident with $v$.
An $(x_1,x_2,\ldots,x_k)$-face is a $k$-face with vertices of degrees $x_1,x_2,\ldots,x_k$.
Two faces $f_1 $ and $f_2$ are called adjacent if they share a common edge.



\section{The Proof of Theorem \ref{thm:main}} \label{sec:premECE}

\subsection{The Structure of Minimum Counterexample} \label{sub:premECE}~~\medskip

Let  $G$ be a minimal counterexample to Theorem \ref{thm:main} such that $|V(G)|+|E(G)|$ is minimum. So, $G$  does not admit any $2$-distance $16$-coloring, but 
any planar graph $G'$ obtained from $G$ with smaller $|V(G')|+|E(G')|$ admits a 2-distance 16-coloring.
Obviously,  $G$ is a connected graph. 


We call a graph $H$ \emph{proper} with respect to $G$ if $H$ is obtained from $G$ by deleting some edges or vertices and adding some edges, ensuring that for every pair of vertices $x_1$ and $x_2$ in $V(G) \cap V(H)$ having distance at most 2 in $G$, they also have distance at most $2$ in $H$.
If $f$ is a 2-distance coloring of such a graph $H$, then $f$ can be extended to the whole graph $G$, provided that each of the remaining uncolored vertices has an available color.\medskip



First, we present some structural results and forbidden configurations for the graph $G$. The proofs of the following lemmas are quite similar: Given a vertex $v\in V(G)$, we construct a graph $H$ consisting of the vertices $V(G-v)$ such that $H$ is proper with respect to $G$. When we have a $2$-distance $16$-coloring of $H$, we show that such a coloring can always be extended to the whole graph $G$ by assigning a color to the removed vertex $v$ when $d_2(v)\leq 15$.


\begin{lemma}\label{lem:deg-geq-3}
For the graph $G$, we have $\delta(G)\geq 3$. 
\end{lemma}
\begin{proof}
Assume that $v$ is a $1$-vertex in $G$. By minimality, the graph $G-v$ has a 2-distance coloring $f$ with $16$ colors. Since $v$ has an available color, the coloring $f$ can be extended to the whole graph $G$, a contradiction.

We next assume that $v$ is a $2$-vertex in $G$. Denote by $x$ and $y$ the neighbours of $v$, and let $G'=G-v+xy$. By minimality, the graph $G'$ has a 2-distance coloring $f$ with $16$ colors. Notice that every vertex in $V(G)$, except $v$, is colored by the same color appears on the same vertex in $G'$. Thus, it suffices to color only $v$. Since $d_2(v)\leq 10$, i.e., $v$ has at most $10$ forbidden colors, we can color $v$ with an available color, a contradiction. Hence, $G$ has no $1$-vertex and $2$-vertex.
\end{proof}

\begin{lemma}\label{lem:3-vertex-no-3-4-neighbour}
A $3$-vertex cannot be adjacent to any $4^-$-vertex. 
\end{lemma}
\begin{proof}
Let $v$ be a $3$-vertex with $N(v)=\{v_1,v_2,v_3\}$. Assume for a contradiction that $v_1$ is a $4^-$-vertex. Let $G'=G-v+\{v_1v_2,v_1v_3\}$. By minimality, the graph $G'$ has a 2-distance coloring $f$ with $16$ colors. Observe that every vertex in $V(G)$, except $v$, is colored with the same color that appears on the same vertex in $G'$. Namely, $G'$ is proper with respect to $G$. Since $d(v_1)+d(v_2)+d(v_3)\leq 14$,   $v$ has at most $14$ forbidden colors, and so we can color $v$ with an available color, a contradiction.
\end{proof}

The following is an immediate consequence of Lemma \ref{lem:3-vertex-no-3-4-neighbour}.

\begin{corollary}\label{cor:vertices-on-5-face}
Let $f$ be a $5$-face in $G$. Then,
\begin{itemize}
\item [$(a)$] there are at most two $3$-vertices on $f$,
\item [$(b)$] if $f$ is incident to two $3$-vertices, then the other vertices on $f$ are $5$-vertices,
\item [$(c)$] if $f$ is incident to exactly one $3$-vertex, then $f$ is incident to at most two $4$-vertices.
\end{itemize}
\end{corollary}

\begin{corollary}\label{cor:number-of-3-vert}
A $5^+$-face $f$ is incident to at most  $ \floor[\big]{\frac{\ell(f)}{2}}$  $3$-vertices.
\end{corollary}

We will now show that a $3$-face does not contain any $3$-vertices, which will be useful throughout the paper.

\begin{lemma}\label{lem:faces-in-3-face}
If $v$ is a $3$-vertex, then $m_3(v)=0$ and  $m_4(v)\leq 1$. 
\end{lemma}
\begin{proof}
Let $v$ be a $3$-vertex in the graph $G$, and let $v_1,v_2,v_3$ be the neighbours of $v$. 
Assume that $v$ is incident to a $3$-face $f_1$, say $f_1=v_1vv_2$.  Let $G'=G-v+v_2v_3$.  By minimality, the graph $G'$ has a 2-distance coloring $f$ with $16$ colors. Observe that every vertex in $V(G)$, except $v$, is colored by the same color appears on the same vertex in $G'$. Since $d_2(v)\leq 13$, we can color $v$ with an available color, a contradiction. Hence, $m_3(v)=0$.

We next assume that $v$ is incident to at least two $4$-faces  $f_1,f_2$, say $f_1=v_1vv_2x$ and $f_2=v_2vv_3y$ for $x,y\in N(v_2)$. Let $G'=G-v+v_1v_3$ (if $v_1v_3\notin E(G)$). Then $G'$ is proper with respect to $G$. Similarly as above, we get a contradiction since $d_2(v)\leq 13$. Thus, $m_4(v)\leq 1$. 
%
\end{proof}

%
%
%
%
%
%
%
%

We now turn our attention into the possible faces incident to a $4$-vertex $v$, where we assume that $v_1,v_2,v_3,v_4$ are the neighbours of $v$ in a clockwise order.

\begin{lemma}\label{lem:4-vertex-has-no-two-3-face}
Let $v$ be a $4$-vertex. Then $m_3(v)\leq 2$. In particular, if  $m_3(v)=2$, then  $m_4(v)=0$ and $n_5(v)=4$.
\end{lemma}
\begin{proof}
Assume that $v$ is incident to three $3$-faces $f_1=v_1vv_2$, $f_2=v_2vv_3$, $f_3=v_3vv_4$. Clearly, we have $d_2(v)\leq 14$. If we set $G'=G-v+v_1v_4$ (assuming $v_1v_4\notin E(G)$), then $G'$ would be proper with respect to $G$.  By minimality, the graph $G'$ has a 2-distance coloring $f$ with $16$ colors. Since $d_2(v)\leq 14$, we can color $v$ with an available color, a contradiction. Hence, $m_3(v)\leq 2$.

Suppose further that $m_3(v)=2$, and let $f_1,f_2$ be $3$-faces incident to $v$. We first show that $m_4(v)=0$. By contradiction, assume that $v$ is incident to a $4$-face, say $f_3=v_1vv_2x$ for $x\in N(v_1)\cap N(v_2)$. Notice that $f_1,f_2$ can be adjacent or not. If $f_1$ and $f_2$ are adjacent, say $f_1=v_2vv_3$ and $f_2=v_3vv_4$, then we set $G'=G-v+v_1v_4$. If $f_1$ and $f_2$ are non-adjacent, say $f_1=v_2vv_3$ and $f_2=v_4vv_1$, then we set $G'=G-v+v_3v_4$. In both cases, we have $d_2(v)\leq 15$, and $G'$ is proper with respect to $G$. Similarly as above, we get a contradiction. Thus, $m_4(v)=0$.
We now show that all neighbours of $v$ are $5$-vertices. By Lemma \ref{lem:3-vertex-no-3-4-neighbour}, all of them are $4^+$-vertices.  Assume to the contrary that  there exists  $v_i\in N(v)$ such that it is a $4$-vertex. If $f_1$ and $f_2$ are adjacent, say $f_1=v_1vv_2$ and $f_2=v_2vv_3$, then we set $G'=G-v+v_2v_4$. If $f_1$ and $f_2$ are non-adjacent, say $f_1=v_1vv_2$ and $f_2=v_3vv_4$, then we set $G'=G-v+\{v_2v_3,v_1v_4\}$. In both cases, we have $d_2(v)\leq 15$, and $G'$ is proper with respect to $G$. Similarly as above, we get a contradiction. Thus, $n_5(v)=4$. 
\end{proof}

\begin{proposition}\label{prop:4-vertex-no-common-m1}
Let $v$ be a $4$-vertex with $m_4(v)=2$ and $n_5(v)=4$. If $v$ is incident to two adjacent faces $f_1,f_2$ such that $f_1$ (resp. $f_2$) is a $3$-face (resp. $5^+$-face), then no edge in $G[N(v)]$ is contained in two $3$-faces of $G$.
\end{proposition}
\begin{proof}
Let $v$ be a $4$-vertex with $m_4(v)=2$ and $n_5(v)=4$. Suppose that $v$ is incident to both a $3$-face $f_1$ and a $5^+$-face $f_2$ such that $f_1$ and $f_2$ are adjacent.
Without loss of generality, we may write $f_1=v_1vv_2$, and assume that $f_2$ is the face containing $v_2,v,v_3$. Since $v$ is incident to two $4$-faces, there exists a common neighbour of each pair  $\{v_3,v_4\}, \{v_1,v_4\}$  in $G-v$. 
Obviously, $v_1v_2$ is the unique edge in $G[N(v)]$ (see Figure \ref{fig:4-adj-faces}).  Assume for a contradiction that $v_1v_2$ is contained in two $3$-faces.  Then we have  $d_2(v)\leq 15$.  If we set  $G'=G-v+\{v_1v_4,v_2v_3\}$, then  $G'$ is proper with respect to $G$.  By minimality, the graph $G'$ has a 2-distance coloring $f$ with $16$ colors. Since $d_2(v)\leq 15$, we can color $v$ with an available color, a contradiction.
\end{proof}

\begin{figure}[htb]
\centering   
\begin{tikzpicture}[scale=1]
\node [nodr] at (0,0) (v) [label=above: {\scriptsize $v$}] {};
\node [nod2] at (1,-1) (v1) [label=below:{\scriptsize $v_1$}] {}
	edge  (v);		
\node [nod2] at (-1,-1) (v2) [label=below:{\scriptsize $v_2$}] {}
	edge [] (v)
	edge [] (v1);
\node [nod2] at (-1,1) (v4) [label=above:{\scriptsize $v_3$}] {}
	edge [] (v);	
\node [nod2] at (1,1) (v5) [label=above:{\scriptsize $v_4$}] {}
	edge [] (v);
\node [nod2] at (-2,-.5) (v13)  {}
	edge  (v2);	
\node [nod2] at (-2,.5) (v53)  {}
	edge  (v4)
	edge [bend right,dotted]  (v13);

\node [nod2] at (2,0) (v21) {}
	edge [] (v1)
	edge [] (v5);
\node [nod2] at (2,-1.5) (v11)  {}
	edge  (v1);	
\node [nod2] at (-2,-1.5) (v22)  {}
	edge  (v2);	
\node [nod2] at (-2,1.5) (v44)  {}
	edge  (v4);	
\node [nod2] at (2,1.5) (v55)  {}
	edge (v5);

\node [nod2] at (0,2) (v41) {}
	edge [] (v4)
	edge [] (v5);


\node at (-1.2,0) (asd)     {\scriptsize $f_2$} ;	
\node at (0,-.6) (asd)     {\scriptsize $f_1$} ;	
\end{tikzpicture}  
\caption{A $4$-vertex with $m_4(v)=2$ and $n_5(v)=4$.}
\label{fig:4-adj-faces}
\end{figure}
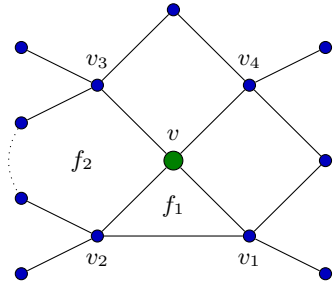

\begin{proposition}\label{prop:4-vertex-no-common}
Let $v$ be a $4$-vertex with $m_3(v)=2$. Then no edge in $G[N(v)]$ is contained in two $3$-faces of $G$.
\end{proposition}
\begin{proof}
Let $v$ be a $4$-vertex with  $m_3(v)=2$. 
All $v_i$'s are $5$-vertices by Lemma \ref{lem:4-vertex-has-no-two-3-face}. 
Let $f_1$ and $f_2$ be $3$-faces incident to $v$ such that $f_1=v_1vv_{2}$. Assume for a contradiction that $v_1v_{2}$ is contained in two $3$-faces. Then we have  $d_2(v)\leq 15$. If $f_1$ and $f_2$ are adjacent, say $f_2=v_2vv_{3}$, then  we set $G'=G-v+v_2v_4$. If $f_2$ and $f_3$ are non-adjacent, say $f_2=v_3vv_{4}$, then we set $G'=G-v+\{v_1v_4,v_2v_3\}$. 
In both cases,  $G'$ is proper with respect to $G$. By minimality, the graph $G'$ has a 2-distance coloring $f$ with $16$ colors. Since $v$ has at most $15$ forbidden colors, we can color $v$ with an available color, a contradiction.
\end{proof}

\begin{lemma}\label{lem:4-vertex-bad-3}
If $v$ is a $4$-vertex and incident to a $(4,4,4^+)$-face, then $v$ is incident to two $5^+$-faces.
\end{lemma}
\begin{proof}
Let $v$ be a $4$-vertex, and 
suppose that $f$ is a $(4,4,4^+)$-face so that $f=v_1vv_2$ with $d(v_1)=4$ and $d(v_2)\geq 4$. By contradiction, assume that $v$ is incident to at most one $5^+$-faces. We then deduce that $v$ is incident to three $4^-$-faces, which implies that $d_2(v)\leq 15$.  In such a case, we set $G'=G-v+\{v_1v_3,v_1v_4\}$ (if $v_1v_3,v_1v_4\notin E(G)$).  By minimality, the graph $G'$ has a 2-distance coloring $f$ with $16$ colors. Clearly, $G'$ is proper with respect to $G$. Because of $d_2(v)\leq 15$, we can color $v$ with an available color, a contradiction. 
\end{proof}



\begin{lemma}\label{lem:4-vertex-semi-bad-3}
If $v$ is a $4$-vertex and incident to a $(4,5,5)$-face, then  $v$ is incident to  a $5^+$-face. 
\end{lemma}
\begin{proof}
Let $v$ be a $4$-vertex, and 
suppose that $f$ is a $(4,5,5)$-face so that $f=v_1vv_2$ with $d(v_1)=d(v_2)=5$. Assume for a contradiction that $v$ is not incident to any $5^+$-face, i.e., all faces incident to $v$ are  $4^-$-faces,  which implies that $d_2(v)\leq 15$.    Let $G'=G-v+\{v_1v_4,v_2v_3\}$ (if $v_1v_4,v_2v_3\notin E(G)$).  By minimality, the graph $G'$ has a 2-distance coloring $f$ with $16$ colors. Clearly, $G'$ is proper with respect to $G$. Since $v$ has at most $15$ forbidden colors, we can color $v$ with an available color, a contradiction. Thus, $v$ is incident to  a $5^+$-face. 
\end{proof}

\begin{proposition}\label{prop:4-vertex-4-5-5-no-4}
Let $v$ be a $4$-vertex and incident to a $(4,5,5)$-face. If $m_3(v)=1$, and $v$ is incident to at most one $5^+$-face, then all neighbours of $v$ are $5$-vertices.
\end{proposition}
\begin{proof}
Let $v$ be a $4$-vertex with  $m_3(v)=1$. 
By Lemma \ref{lem:3-vertex-no-3-4-neighbour}, we have $n_3(v)=0$. Suppose that $f$ is a $(4,5,5)$-face such that $f=v_1vv_2$ with $d(v_1)=d(v_2)=5$. Let $v$ be incident to at most one $5^+$-face. Assume for a contradiction that $v$ has a $4$-neighbour, say $v_3$. Since $m_3(v)=1$, and $v$ is incident to at most one $5^+$-face, we can say that $v$ is incident to at least two $4$-faces, say $f_2,f_3$. Clearly, we have $d_2(v)\leq 15$. If $f_2$ and $f_3$ are adjacent, then  we set $G'=G-v+\{v_2v_3,v_1v_4\}$. If $f_2$ and $f_3$ are non-adjacent, then we set $G'=G-v+\{v_2v_3,v_3v_4\}$. 
In each case,  $G'$ is proper with respect to $G$. By minimality, the graph $G'$ has a 2-distance coloring $f$ with $16$ colors. Since $d_2(v)\leq 15$, we can color $v$ with an available color, a contradiction. Thus, all neighbours of $v$ are $5$-vertices.
\end{proof}

We will now address some properties of $5$-vertices, with the assumption that $v_1, v_2, \ldots, v_5$ are the neighbours of $v$ in clockwise order. Furthermore, if $m_3(v) = 4$, we consider the $3$-faces incident to $v$ as $f_1 = v_1vv_2$, $f_2 = v_2vv_3$, $f_3 = v_3vv_4$, and $f_4 = v_4vv_5$.

\begin{lemma}\label{lem:faces-on-5-vertex}
Let $v$ be a $5$-vertex. Then $m_3(v)\leq 4$ and $n_3(v)\leq 4$. 
\end{lemma}
\begin{proof}
Assume to the contrary that all faces incident to $v$ are $3$-faces, i.e, $m_3(v)=5$. Let $G'=G-v$.  By minimality, the graph $G'$ has a 2-distance coloring $f$ with $16$ colors.  Clearly, $G'$ is proper with respect to $G$. Since $v$ has at most $15$ forbidden colors, we can color $v$ with an available color, a contradiction. Thus, $m_3(v)\leq 4$.

Let us now assume that all neighbours of $v$ are $3$-vertices, i.e, $n_3(v)=5$. If we set $G'=G-v+\{v_1v_2,v_2v_3,v_3v_4,v_4v_5,v_5v_1\}$ (if $v_iv_{i+1}\notin E(G)$ in the cyclic fashion), then  $G'$ is proper with respect to $G$.
Similarly as above, we get a contradiction since $d_2(v)\leq 15$. 
Thus, $n_3(v)\leq 4$. 
%
%
\end{proof}


\begin{proposition}\label{prop:  v has two 4 neighbour}
Let $v$ be a $5$-vertex. If $m_3(v)=n_3(v)= 2$, then $v$ has at  most one $4$-neighbour. 
\end{proposition}
\begin{proof}
Let $m_3(v)=n_3(v)= 2$.  Since  $n_3(v)= 2$, the $3$-faces incident to $v$, say $f_1$ and $f_2$, must be adjacent by Lemma \ref{lem:faces-in-3-face}. We write  $f_1=v_1vv_2$ and $f_2=v_2vv_3$. By contradiction, assume that $v$ has two $4$-neighbours. This infer that $d_2(v)\leq 15$. If we set $G'=G-v+\{v_3v_4,v_4v_5,v_5v_1\}$, then $G'$ is proper with respect to $G$. By minimality, the graph $G'$ has a 2-distance coloring $f$ with $16$ colors. Due to $d_2(v)\leq 15$, we can color $v$ with an available color, a contradiction. 
\end{proof}

\begin{lemma}\label{lem:5-vertex-three-4-neigh}
Let $v$ be a $5$-vertex with  $n_3(v)= 1$. 
\begin{itemize}
\item[$(a)$] If $m_3(v)=2$, then $v$ has at  most three $4$-neighbours. 
\item[$(b)$] If $m_3(v)=2$, and $v$ has a $4$-neighbour, then  $v$ is incident to a $5^+$-face. 
\item[$(c)$] If $m_3(v)=3$, then  $v$ has at  most one $4$-neighbours and incident to at least one $5^+$-face.
\item[$(d)$] If $m_3(v)=3$, and $v$ has a $4$-neighbour, then  $v$ is incident to two $5^+$-face.
\end{itemize}
\end{lemma}

\begin{proof}
Let $v_5$ be $3$-neighbour of $v$. By Lemma \ref{lem:faces-in-3-face}, $v_5$ does not belong to any $3$-face.

(a). Suppose that $m_3(v)=2$, and let $f_1$ and $f_2$ be $3$-faces incident to $v$. Assume for a contradiction that $v$ has four $4$-neighbours, and so $v_1,v_2,v_3,v_4$ are $4$-vertices. If $f_1$ and $f_2$ are adjacent, say $f_1=v_1vv_2$ and $f_2=v_2vv_3$, then we set $G'=G-v+\{v_3v_4,v_4v_5,v_1v_5\}$. If $f_1$ and $f_2$ are non-adjacent, say $f_1=v_1vv_2$ and $f_2=v_3vv_4$, then we set $G'=G-v+\{v_2v_3,v_4v_5,v_1v_5\}$. In both cases, we have $d_2(v)\leq 15$ and $G'$ is proper with respect to $G$.  By minimality, the graph $G'$ has a 2-distance coloring $f$ with $16$ colors.  Since $d_2(v)\leq 15$, we can color $v$ with an available color, a contradiction.

(b). Let $m_3(v)=2$ and let $v$ have a $4$-neighbour. By contradiction, assume that $v$ is not incident to any $5^+$-face, i.e., $v$ has three $4$-faces. Let $f_1$ and $f_2$ be $3$-faces incident to $v$.  
If $f_1$ and $f_2$ are adjacent, say $f_1=v_1vv_2$ and $f_2=v_2vv_3$, then we set $G'=G-v+\{v_1v_5,v_2v_5,v_4v_5\}$, where we note that $v_5$ shares a common neighbour, except $v$, with each of $v_1$ and $v_4$, since $v$ has three $4$-faces.
If $f_1 $ and $f_2$ are non-adjacent, say $f_1=v_1vv_2$ and $f_2=v_3vv_4$, then we set $G'=G-v+\{v_2v_3,v_4v_5,v_1v_5\}$. In both cases, we have $d_2(v)\leq 15$ and $G'$ is proper with respect to $G$.  Similarly as above, we get a contradiction. Thus, $v$ is incident to a $5^+$-face. 


(c). Let $m_3(v)=3$ and let $f_1,f_2,f_3$ be $3$-faces incident to $v$. We may write $f_1=v_1vv_2$, $f_2=v_2vv_3$, and $f_3=v_3vv_4$, since $v_5$ is not incident to any $3$-face by  Lemma \ref{lem:faces-in-3-face}. Assume for a contradiction that $v$ has two $4$-neighbours or all faces incident to $v$ are $4^-$-faces. In both cases, we  have $d_2(v)\leq 15$. Let $G'=G-v+\{v_4v_5,v_5v_1\}$.  Clearly, $G'$ is proper with respect to $G$. Similarly as above, we get a contradiction.  Thus, $n_4(v)\leq 1$, and $v$ is incident to at least one $5^+$-face. 

(d). Let $m_3(v)=3$ and let $f_1,f_2,f_3$ be $3$-faces incident to $v$. We may write $f_1=v_1vv_2$, $f_2=v_2vv_3$, and $f_3=v_3vv_4$, since $v_5$ is not incident to any $3$-face by  Lemma \ref{lem:faces-in-3-face}. Suppose that $v$ has a $4$-neighbour. If $v$ is incident to at most one $5^+$-face, then we would have $d_2(v)\leq 15$.  In such a case, if we set $G'=G-v+\{v_4v_5,v_5v_1\}$, then $G'$ would be proper with respect to $G$. Similarly as above, we get a contradiction. Thus, $v$ is incident to two $5^+$-faces. 
\end{proof}

\begin{lemma}\label{lem:5-vertex-n3-0}
Let $v$ be a $5$-vertex with $n_3(v)= 0$. 
\begin{itemize}
\item[$(a)$] If $m_3(v)=2$, and all neighbours of $v$ are $4$-vertices, then  $v$ is incident to three $5^+$-faces.
\item[$(b)$] If $m_3(v)=3$, then  $v$ has at  most three $4$-neighbours.
\end{itemize}
\end{lemma}

\begin{proof} 
%
(a). Let $m_3(v)=2$ and $n_4(v)=5$. Assume to the contrary that $v$ is incident to at most two $5^+$-face. This implies that $v$ is incident to a $4$-face. Clearly, we have $d_2(v)\leq 15$.
If the $3$-faces incident to $v$ are adjacent, say $f_1=v_1vv_2$, $f_2=v_2vv_3$, then we write $G'=G-v+\{v_3v_4,v_4v_5,v_5v_1\}$. If the $3$-faces incident to $v$ are not adjacent, say $f_1=v_1vv_2$, $f_2=v_3vv_4$, then we write $G'=G-v+\{v_2v_3,v_4v_5,v_5v_1\}$. In both cases, $G'$ is proper with respect to $G$. By minimality, the graph $G'$ has a 2-distance coloring $f$ with $16$ colors. Since  $d_2(v)\leq 15$, we can color $v$ with an available color, a contradiction. Thus, $v$ is incident to three $5^+$-faces.


(b). Let $m_3(v)=3$ and  let $f_1,f_2,f_3$ be $3$-faces incident to $v$. By contradiction, suppose that  $v$ has four $4$-neighbours. Then we have $d_2(v)\leq 15$.
If the $3$-faces  incident to $v$ are forming $f_1=v_1vv_2$, $f_2=v_2vv_3$, and $f_3=v_4vv_5$, then we write $G'=G-v+\{v_3v_4,v_5v_1\}$. If the $3$-faces incident to $v$ are forming $f_1=v_1vv_2$, $f_2=v_2vv_3$, and $f_3=v_3vv_4$, then we write $G'=G-v+\{v_2v_5,v_2v_4\}$ where we should note that one of $v_2$ or $v_3$ must be a $4$-vertex since $v$ has four $4$-neighbours, let us assume it is $v_2$ by symmetry. In both cases, $G'$ is proper with respect to $G$.
Similarly as above, we get a contradiction. Thus, $n_4(v)\leq 3$. 
\end{proof}

\begin{lemma}\label{lem:5-vertex-n3-0-m3-4}
Let $v$ be a $5$-vertex with  $n_3(v)= 0$ and $m_3(v)=4$. 
\begin{itemize}
\item[$(a)$] The vertex $v$ has at  most one $4$-neighbour.
\item[$(b)$] If  $v$ has a $4$-neighbour, $v$ is incident to a $5^+$-face.
\item[$(c)$] If $v$ is incident to a $4$-face, there is no edge $v_iv_{i+1}$ for $i\in [4]$ contained in two $3$-faces.
\item[$(d)$] If $v$ is incident to a $5^+$-face, there exists at most one edge $v_iv_{i+1}$ for $i\in [4]$ contained in two $3$-faces. 
\item[$(e)$] If $v$ is incident to a $5^+$-face, and there exists an edge $v_iv_{i+1}$ for $i\in [4]$ contained in two $3$-faces, then all neighbour of $v$ are $5$-vertices.
\end{itemize}
\end{lemma}

\begin{proof}
Let $f_1,f_2,f_3,f_4$ be $3$-faces incident to $v$, say $f_1=v_1vv_2$, $f_2=v_2vv_3$, $f_3=v_3vv_4$, and $f_3=v_4vv_5$.

(a). By contradiction, assume that $v$ has two $4$-neighbours, and so $d_2(v)\leq 15$. If we set $G'=G-v+\{v_5v_1\}$, then $G'$ is proper with respect to $G$. By minimality, the graph $G'$ has a 2-distance coloring $f$ with $16$ colors. Since $v$ has at most $15$ forbidden colors, we can color $v$ with an available color, a contradiction. Thus, $n_4(v)\leq 1$. 

(b). Suppose that $v$ has a $4$-neighbour, and  $v$ is incident to a $4$-face. Then we have  $d_2(v)\leq 15$. Similarly as above, we get a contradiction.

(c). Suppose that $v$ is incident to a $4$-face. If there is an edge $v_iv_{i+1}$ for $i\in [4]$ contained in two $3$-faces, then we have again $d_2(v)\leq 15$. Similarly as above, we get a contradiction.

(d). Assume that  $v$ is incident to a $5^+$-face. If there exist two edges $v_iv_{i+1}$  and $v_jv_{j+1}$ for $i,j\in [4]$ with $i\neq j$  contained in two $3$-faces, then we have again $d_2(v)\leq 15$. Similarly as above, we get a contradiction.

(e). Suppose that $v$ is incident to a $5^+$-face, and there exists an edge $v_iv_{i+1}$ for $i\in [4]$ contained in two $3$-faces. If $v$ has also a $4$-neighbour, then we have again $d_2(v)\leq 15$. Similarly as above, we get a contradiction. 
\end{proof}



Let us define some special vertices. Suppose that $v$ is a $5$-vertex, and let $v_1, v_2, \ldots, v_5$ be the neighbours of $v$ in clockwise order. If $m_3(v) = 4$, and $v$ is incident to a $4$-face (resp. a $5^+$-face), we refer to $v$ as a \emph{bad vertex} (resp. a \emph{semi-bad vertex}). The other special vertices are as follows: \medskip


We call $v$ as a \emph{strong vertex}  if 
\begin{itemize}
\item $m_3(v)=3$,
\item the $3$-faces incident to $v$ are forming $f_1=v_1vv_2$, $f_2=v_2vv_3$, and $f_3=v_4vv_5$,
\item $v_2$ is a bad or semi-bad vertex,
\item $v$ is incident to a $5^+$-face.
\end{itemize}

We call $v$ as a \emph{good vertex} if 
\begin{itemize}
\item $m_3(v)= 3$,
\item the $3$-faces incident to $v$ are forming $f_1=v_1vv_2$, $f_2=v_2vv_3$, and $f_3=v_3vv_4$,
\item $v_2$ is a semi-bad vertex,
\item each of $v_1v_2$ and $v_2v_3$ is contained in two $3$-faces. 
\end{itemize}

We call $v$ as a \emph{support vertex}  if 
\begin{itemize}
\item $m_3(v)=2$,
\item $v$ is incident to two adjacent $3$-faces forming $f_1=v_1vv_2,$ and $f_2=v_2vv_3$,
\item $v_2$ is a bad or semi-bad vertex.
\end{itemize}

We illustrate these special vertices in  Figures \ref{fig:bad5} and  \ref{fig:others} with the assistance of the results in Remark \ref{rem:specil vertices} and Lemma \ref{lem:deger}. Notice that all $v_i$'s, in Figure \ref{fig:bad5}(a), are $5$-vertices while $v_2,v_3,v_4$, in Figure \ref{fig:bad5}(b), are  $5$-vertices by Lemma \ref{lem:deger}(a). Additionally, one of $v_1$ or $v_5$ in Figure \ref{fig:bad5}(b) is  a $5$-vertex and the other is a $4^+$-vertex by Remark \ref{rem:specil vertices}(b).

Moving on to the other special vertices, in Figures \ref{fig:others}(a),(b) and (c), the vertex $v_2$ is either a bad or a semi-bad vertex, implying that $d(v_2) = 5$ and $m_4(v_2) = 4$. Furthermore, in Figures \ref{fig:others}(a) and \ref{fig:others}(c), $v_3$ is a $5$-vertex (resp. $4^+$-vertex) if $v_2$ is a bad (resp. semi-bad) vertex by Remark \ref{rem:specil vertices}.  In Figure \ref{fig:others}(b), all neighbours of $v_2$ are $5$-vertices by Lemma \ref{lem:5-vertex-n3-0-m3-4}(e).

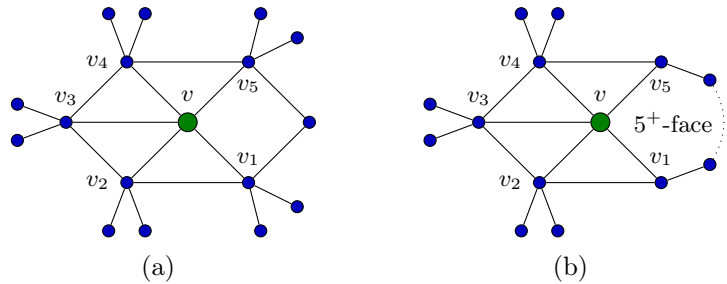
\begin{figure}[htb]
\centering   
\subfigure[]{
\label{fig:bad5-a}
\begin{tikzpicture}[scale=.8]
\node [nodr] at (0,0) (v) [label=above: {\scriptsize $v$}] {};
\node [nod2] at (1,-1) (v1) [label=above:{\scriptsize $v_1$}] {}
	edge  (v);		
\node [nod2] at (-1,-1) (v2) [label=left:{\scriptsize $v_2$}] {}
	edge [] (v)
	edge [] (v1);
\node [nod2] at (-2,0) (v3) [label=above:{\scriptsize $v_3$}]  {}
	edge [] (v)
	edge [] (v2);
\node [nod2] at (-1,1) (v4) [label=left:{\scriptsize $v_4$}] {}
	edge [] (v)
	edge [] (v3);	
\node [nod2] at (1,1) (v5) [label=below:{\scriptsize $v_5$}] {}
	edge [] (v)
	edge [] (v4);
\node [nod2] at (1.2,-1.8) (v11)  {}
	edge  (v1);	
\node [nod2] at (1.8,-1.4) (v12)  {}
	edge  (v1);	
\node [nod2] at (2,0) (v13)  {}
	edge  (v1)
	edge  (v5);	
\node [nod2] at (-.7,-1.8) (v21) {}
	edge [] (v2);
\node [nod2] at (-1.3,-1.8) (v22) {}
	edge [] (v2);
\node [nod2] at (-2.8,-.3) (v31) {}
	edge [] (v3);
\node [nod2] at (-2.8,.3) (v32) {}
	edge [] (v3);
\node [nod2] at (-.7,1.8) (v41) {}
	edge [] (v4);
\node [nod2] at (-1.3,1.8) (v42) {}
	edge [] (v4);	
\node [nod2] at (1.2,1.8) (v51)  {}
	edge  (v5);	
\node [nod2] at (1.8,1.4) (v52)  {}
	edge  (v5);				
\end{tikzpicture}  }\hspace*{1cm}
\subfigure[]{
\label{fig:bad5-b}
\begin{tikzpicture}[scale=.8]
\node [nodr] at (0,0) (v) [label=above: {\scriptsize $v$}] {};
\node [nod2] at (1,-1) (v1) [label=above:{\scriptsize $v_1$}] {}
	edge  (v);		
\node [nod2] at (-1,-1) (v2) [label=left:{\scriptsize $v_2$}] {}
	edge [] (v)
	edge [] (v1);
\node [nod2] at (-2,0) (v3) [label=above:{\scriptsize $v_3$}]  {}
	edge [] (v)
	edge [] (v2);
\node [nod2] at (-1,1) (v4) [label=left:{\scriptsize $v_4$}] {}
	edge [] (v)
	edge [] (v3);	
\node [nod2] at (1,1) (v5) [label=below:{\scriptsize $v_5$}] {}
	edge [] (v)
	edge [] (v4);
\node [nod2] at (1.8,-.7) (v13)  {}
	edge  (v1);

\node [nod2] at (-.7,-1.8) (v21) {}
	edge [] (v2);
\node [nod2] at (-1.3,-1.8) (v22) {}
	edge [] (v2);

\node [nod2] at (-2.8,-.3) (v31) {}
	edge [] (v3);
\node [nod2] at (-2.8,.3) (v32) {}
	edge [] (v3);

\node [nod2] at (-.7,1.8) (v41) {}
	edge [] (v4);
\node [nod2] at (-1.3,1.8) (v42) {}
	edge [] (v4);	

\node [nod2] at (1.8,.7) (v53)  {}
	edge  (v5)
	edge [bend left,dotted]  (v13);	

\node at (1.2,0) (asd)     {\scriptsize $5^+$-face} ;	
\end{tikzpicture}  }
\caption{$(a)$ A bad vertex. \ $(b)$ A semi-bad vertex.}
\label{fig:bad5}
\end{figure}

\begin{figure}[htb]
\centering   
\subfigure[]{
\label{fig:bad5-strong}
\begin{tikzpicture}[scale=.8]
\node [nodr] at (-2,0) (v) [label=above:{\scriptsize $v$}]  {};
\node [nod2] at (-1,-1) (v1) [label=above:{\scriptsize $v_3$}] {}
	edge [] (v);
\node [nod2] at (0,0) (v2) [label=above: {\scriptsize $v_2$}] {}
	edge  (v)
	edge  (v1);
\node [nod2] at (-1,1) (v3) [label=below:{\scriptsize $v_1$}] {}
	edge [] (v)
	edge [] (v2);	
\node [nod2] at (1,1) (u5)  {}
	edge [] (v2)
	edge [] (v3);

	
\node [nod2] at (-3.2,-.6) (v31) [label=below:{\scriptsize $v_4$}] {}
	edge [] (v);
\node [nod2] at (-3.2,.6) (v32) [label=above:{\scriptsize $v_5$}] {}
	edge [] (v31)
	edge [] (v);
\node [nod2] at (-.7,1.8) (v41) {}
	edge [] (v3);
\node [nod2] at (-1.3,1.8) (v42) {}
	edge [] (v3);	
\node at (0,-1.8) (asd)     {} ;				
\end{tikzpicture}  }\hspace*{.5cm}
\subfigure[]{
\label{fig:bad5-good}
\begin{tikzpicture}[scale=.8]
\node [nodr] at (-2,0) (v) [label=above:{\scriptsize $v$}]  {};
\node [nod2] at (-1,-1) (v1) [label=above:{\scriptsize $v_3$}] {}
	edge [] (v);
\node [nod2] at (0,0) (v2) [label=above: {\scriptsize $v_2$}] {}
	edge  (v)
	edge  (v1);
\node [nod2] at (1,-1) (u1)  {}
	edge  (v1)
	edge  (v2);		
\node [nod2] at (-1,1) (v3) [label=below:{\scriptsize $v_1$}] {}
	edge [] (v)
	edge [] (v2);	
\node [nod2] at (1,1) (u5)  {}
	edge [] (v2)
	edge [] (v3);

\draw[dotted] (u5) .. controls (2.5,0) .. (u1);
\node at (1.2,0) (asd)     {\scriptsize $5^+$-face} ;	
	
\node [nod2] at (-.7,-1.8) (v21) {}
	edge [] (v1);
\node [nod2] at (-2.1,-1.3) (v31) [label=below:{\scriptsize $v_4$}] {}
	edge [] (v)
	edge [] (v1);
\node [nod2] at (-3.2,0) (v32) [label=above:{\scriptsize $v_5$}] {}
	edge [] (v);
\node [nod2] at (-.7,1.8) (v41) {}
	edge [] (v3);
\node [nod2] at (-1.3,1.8) (v42) {}
	edge [] (v3);				
\end{tikzpicture}  }\hspace*{.5cm}
\subfigure[]{
\label{fig:bad5-support}
\begin{tikzpicture}[scale=.8]
\node [nodr] at (-2,0) (v) [label=above:{\scriptsize $v$}]  {};
\node [nod2] at (-1,-1) (v1) [label=above:{\scriptsize $v_3$}] {}
	edge [] (v);
\node [nod2] at (0,0) (v2) [label=below: {\scriptsize $v_2$}] {}
	edge  (v)
	edge  (v1);	
\node [nod2] at (-1,1) (v3) [label=below:{\scriptsize $v_1$}] {}
	edge [] (v)
	edge [] (v2);	
\node [nod2] at (1,1) (u5)  {}
	edge [] (v2)
	edge [] (v3);
	
\node [nod2] at (-3.2,-.6) (v31) [label=below:{\scriptsize $v_4$}] {}
	edge [] (v);
\node [nod2] at (-3.2,.6) (v32) [label=above:{\scriptsize $v_5$}] {}
	edge [] (v);
\node [nod2] at (-.7,1.8) (v41) {}
	edge [] (v3);
\node [nod2] at (-1.3,1.8) (v42) {}
	edge [] (v3);	
\node at (0,-1.8) (asd)     {} ;					
\end{tikzpicture}  }
\caption{$(a)$ A strong vertex. \ $(b)$ A good vertex. \ $(c)$ A support vertex.}
\label{fig:others}
\end{figure}
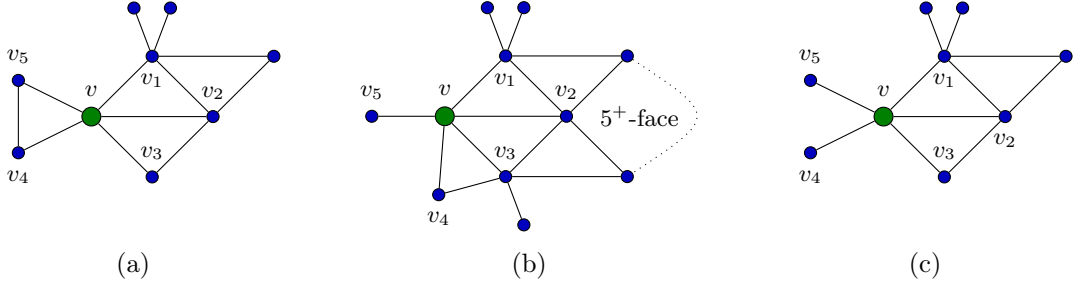

Our next aim is to find the structural properties of the special vertices defined earlier. Recall that a bad and a semi-bad vertex has $m_3(v)=4$. Therefore, based on this assumption, we can conclude that the $3$-faces incident to $v$ are $f_1=v_1vv_2$, $f_2=v_2vv_3$, $f_3=v_3vv_4$, and $f_3=v_4vv_5$. Additionally, a bad or a semi-bad vertex cannot have any $3$-neighbour by Lemma \ref{lem:faces-in-3-face}.

\begin{remark}\label{rem:specil vertices} 
Let $v$ be a $5$-vertex.
\begin{itemize}
\item[$(a)$] If $v$ is a bad vertex, then all neighbours of $v$ are $5$-vertices, and no edge $v_iv_{i+1}$ for $i\in [4]$ is contained in two $3$-faces.

\item[$(b)$] If $v$ is a semi-bad vertex, then $n_3(v)=0$ and $n_5(v)\geq 4$. 
%
%
\end{itemize}
\end{remark}

\begin{proof} 
(a). Suppose that $v$ is a bad vertex. It follows from Lemma \ref{lem:faces-in-3-face}  that all $v_i$'s are $4^+$-vertices. By contradiction, assume that $v$ has a $4$-neighbour. Since $v,v_1,v_5$ are contained in a $4$-face, we have $d_2(v)\leq 15$. If we set $G'=G-v+\{v_5v_1\}$, then  $G'$ is proper with respect to $G$. By minimality, the graph $G'$ has a 2-distance coloring $f$ with $16$ colors. Since $v$ has at most $15$ forbidden colors, we can color $v$ with an available color, a contradiction. Thus, $n_5(v)=5$. The remaining of the claim follows from Lemma \ref{lem:5-vertex-n3-0-m3-4}(c).


(b). Suppose that $v$ is a semi-bad vertex. Note that, by Lemma \ref{lem:faces-in-3-face} and Lemma \ref{lem:5-vertex-n3-0-m3-4}(a), $v$ has neither a $3$-neighbour nor two $4$-neighbours.
\end{proof}

%

\begin{remark}\label{rem:bad-no-bad}
A bad vertex cannot have any bad or semi-bad neighbour.
\end{remark}
\begin{proof}
Let $v$ be a bad vertex. 
Since there is no edge $v_iv_{i+1}$ for $i\in [4]$ contained in two $3$-faces by Lemma \ref{lem:5-vertex-n3-0-m3-4}(c),  none of $v_i$'s can be incident to four $3$-faces, i.e., we have $m_3(v_i)\leq 3$ for each $v_i\in N(v)$.  This implies that no neighbour of $v$ is a bad or a semi-bad vertex. So, the claim holds.
\end{proof}

\begin{lemma}\label{lem:deger}
Let $v$ be a bad or a semi-bad vertex (see Figure \ref{fig:bad5}). Then the followings hold.
\begin{itemize}
\item[$(a)$] $v_2,v_3,v_4$ are $5$-vertices.
\item[$(b)$] If $v$ is a bad vertex, then $v_3$ is either a support or a strong vertex.
\item[$(c)$] If $v$ is a semi-bad vertex, then $v_3$ is  either a support or a strong or a good vertex.
\end{itemize}
\end{lemma}
\begin{proof} 
(a). By Remark \ref{rem:specil vertices}(a), all $v_i$'s are $5$-vertices when $v$ is a bad vertex. So we consider only semi-bad case. Let $v$ be a semi-bad vertex. It follows from Lemma \ref{lem:faces-in-3-face} that no neighbour of $v$ is a $3$-vertex, i.e., all $v_i$'s  are $4^+$-vertices. Assume for a contradiction that there exists $v_i \in \{v_2,v_3,v_4\}$ such that it is a $4$-vertex, without loss of generality, say $v_i=v_2$ and let $x\in N(v_2)\setminus \{v,v_1,v_3\}$. We observe that $d_2(v_2)\leq 15$. If we set $G'=G-v_2+vx$, then  $G'$ is proper with respect to $G$. By minimality, the graph $G'$ has a 2-distance coloring with $16$ colors. Since $d_2(v_2)\leq 15$, we can color $v_2$ with an available color, a contradiction. Thus, $v_2,v_3,v_4$ are $5$-vertices.

(b). Let $v$ be a bad vertex. By (a), $v_3$ is a $5$-vertex, and so we denote by $v_4,v,v_2,x,y$  the neighbours of $v_3$ in clockwise order.  If $m_3(v_3)=2$, then $v_3$ would be a support vertex. Notice that the case $m_3(v_3)=4$ is not possible by Remark \ref{rem:bad-no-bad}.  Thus, we may suppose that $m_3(v_3)=3$, and so $xy\in E(G)$ by Remark \ref{rem:specil vertices}(a). 
If all faces incident to $v_3$ are $4^-$-faces, then we have $d_2(v_3)\leq 15$ since $v$ has two common neighbours with each of $v_2,v_4$. In such a case, we set $G'=G-v_3+\{v_2x,yv_4\}$.  Similarly as above, we get a contradiction. We therefore conclude that $v_3$ is incident to a $5^+$-face, and so $v_3$ becomes a strong vertex.

(c). Let $v$ be a semi-bad vertex. The vertex $v_3$ is a $5$-vertex by (a), and denote by $v_4,v,v_2,x,y$ the neighbours of $v_3$ in clockwise order.  If $m_3(v_3)=2$, then $v_3$ is a support vertex by its definition. So, we may suppose that $m_3(v_3)\geq 3$. We remark that if $m_3(v_3)=4$, then $v_3$ becomes a bad or a semi-bad vertex, however, both $v_2v$ and $vv_4$ would be contained in two $3$-faces, which contradicts Lemma \ref{lem:5-vertex-n3-0-m3-4}(c)-(d). Therefore we may assume that  $m_3(v_3)= 3$. If none of the edges $v_2v_3, v_3v_4$ is contained in two $3$-faces, then similarly as above, we conclude that $v_3$ is incident to a $5^+$-face, and so $v_3$ is a strong vertex.
If only one of the edges $v_2v_3, v_3v_4$ is contained in two $3$-faces, by symmetry, we may assume it is $v_2v_3$, then $v_3$ becomes a good vertex. \medskip
\end{proof}

\begin{lemma}\label{lem:strong-properties}
Let $v$ be a strong vertex with $3$-faces $f_1=v_1vv_2$, $f_2=v_2vv_3$ and $f_3=v_4vv_5$. 
\begin{itemize}
\item[$(a)$] If $v$ has two $4$-neighbours, then $v$ is incident to two $5^+$-face.
\item[$(b)$] If $v_2$ has two common neighbours with each of $v_1,v_3$, then 
either $v$ has no $4$-neighbour or $v$ is incident to  two $5^+$-face. 
\end{itemize}
\end{lemma}
\begin{proof}
(a). Let $v$ be a strong vertex with $3$-faces $f_1=v_1vv_2$, $f_2=v_2vv_3$ and $f_3=v_4vv_5$. Then $v_2$ is a bad or a semi-bad vertex, which implies that at least one of $v_1v_2$ and $v_2v_3$ is contained in two $3$-faces. Suppose that $v$ has two $4$-neighbours.  By contradiction, assume that $v$ is incident to at most one $5^+$-face. Thus, we have  $d_2(v)\leq 15$. If we set $G'=G-v+\{v_3v_4,v_1v_5\}$, then $G'$ would be proper with respect to $G$.
By minimality, the graph $G'$ has a 2-distance coloring with $16$ colors. Since $v$ has at most $15$ forbidden colors, we can color $v$ with an available color, a contradiction. Hence, $v$ is incident to two $5^+$-faces.

(b). 
Suppose that $v_2$ has two common neighbours with each of $v_1,v_3$.  By contradiction, assume that $v$ has a $4$-neighbour and $v$ is incident to a $4$-face. Clearly, we have  $d_2(v)\leq 15$. 
Similarly as above, we get a contradiction. So the claim holds.
\end{proof}

\begin{proposition}\label{prop:good-no-two-semi-bad}
Let $v$ be a good vertex with $3$-faces $f_1=v_1vv_2$, $f_2=v_2vv_3$, and $f_3=v_3vv_4$. If $n_4(v)=1$, and $v$ is incident to exactly two $5^+$-faces, then only one of $v_2,v_3$ can be semi-bad vertex. 
\end{proposition}
\begin{proof}
Suppose that $n_4(v)=1$, and $v$ is incident to exactly two $5^+$-faces. It follows from the definition of good vertex, each of $v_1v_2$, $v_2v_3$ is contained in two $3$-faces. By contradiction, assume that both $v_2$ and $v_3$ are semi-bad vertices. It then follows from Lemma \ref{lem:5-vertex-n3-0-m3-4}(d) that $v_3v_4$ is contained in two $3$-faces. This implies that  $d_2(v)\leq 15$. By Lemma \ref{lem:deger}(a), the vertices $v_1,v_2,v_3,v_4$ are $5$-vertices. We then deduce that $v_5$ is a $4$-vertex since $n_4(v)=1$. Thus, if we set $G'=G-v+\{v_1v_5,v_4v_5\}$, then  $G'$ would be proper with respect to $G$.
By minimality, the graph $G'$ has a 2-distance coloring with $16$ colors.  Since $v$ has at most $15$ forbidden colors, we can color $v$ with an available color, a contradiction.
\end{proof}

\begin{lemma}\label{lem:good-properties}
Let $v$ be a good vertex with $3$-faces $f_1=v_1vv_2$, $f_2=v_2vv_3$ and $f_3=v_3vv_4$ such that $v_2$ is a semi-bad vertex. Then the followings hold.
\begin{itemize}
\item[$(a)$] $v$ has no $3$-neighbour. 
\item[$(b)$] $v$ has at most one $4$-neighbour.
\item[$(c)$] $v$ is incident to a $5^+$-face. 
\item[$(d)$] If $v$ has a $4$-neighbour, then $v$ is incident to two $5^+$-face. 
\item[$(e)$] If $v_3$ is also a semi-bad vertex, then $v$ is incident to two $5^+$-face.
\end{itemize}
\end{lemma}
\begin{proof} Suppose that $v$ is a good vertex with $3$-faces $f_1=v_1vv_2$, $f_2=v_2vv_3$ and $f_3=v_3vv_4$ such that $v_2$ is a semi-bad vertex. By the definition of good vertex, each of  $v_1v_2$, $v_2v_3$ is contained in two $3$-faces. Denote by $f_1,f_2,\ldots,f_5$ the faces incident to $v$ in a clockwise order. 

(a). By contradiction, assume that $v$ has a $3$-neighbour, which should be $v_5$ by Lemma \ref{lem:faces-in-3-face}. Clearly, we have $d_2(v)\leq 15$. If we set $G'=G-v+\{v_4v_5,v_5v_1\}$, 
then $G'$ is proper with respect to $G$. By minimality, the graph $G'$ has a 2-distance coloring with $16$ colors. Since  $d_2(v)\leq 15$, we can color $v$ with an available color, a contradiction.  Thus, $v$ has no $3$-neighbour.

(b). Assume to the contrary that $v$ has two $4$-neighbours. Thus we have  $d_2(v)\leq 15$.  Notice that $v_1,v_2,v_3$ are $5$-vertices by Lemma \ref{lem:deger}(a). 
It then follows that $v_4$ and $v_5$ are $4$-vertices by the assumption.
If we set $G'=G-v+\{v_4v_5,v_5v_1\}$, then $G'$ is proper with respect to $G$.  Similarly as above, we get a contradiction. So, $v$ has at most one $4$-neighbour.

(c). Suppose that $v$ is incident to no $5^+$-face, i.e., $f_4,f_5$ are $4$-faces. This particularly implies that $v_5$ has a common neighbour, except $v$, with each of $v_1,v_4$.  Clearly, we have $d_2(v)\leq 15$.  
Denote by $v_6$ (resp., $v_7$) the common neighbour of $v_1$ and $v_2$ (resp., $v_2$ and $v_3$) other than $v$.
Let $G'$ be the graph obtained from $G$ by deleting vertices $v,v_2$ and adding edges $v_1v_3$, $v_4v_5$, $v_6v_7$. Observe that every vertex in $V(G)$, except $v$ and $v_2$, is colored with the same color that appears on the same vertex in $G'$. Namely, $G'$ is proper with respect to $G$. Thus, it suffices to color only $v$ and $v_2$. Since $d_2(v_2)\leq 16$ and $v$ is uncolored vertex, we say that the vertex $v_2$ has at most $15$ forbidden colors. Therefore we can color $v_2$ with an available color. In addition, we can color $v$ with an available color since  $d_2(v)\leq 15$, a contradiction.  Thus, $v$ is incident to a $5^+$-face.

(d).  Suppose that $v$ has a $4$-neighbour, and assume for a contradiction that $v$ is incident to at most one $5^+$-face. This implies that at least one of $f_4,f_5$ is a $4$-face. Without loss of generality, assume that $f_5$ is a $4$-face forming $f_5=vv_1xv_5$ for $x\in N(v_1)\cap N(v_5)$.  Clearly, we have $d_2(v)\leq 15$. 
Let $G'$ be the graph obtained from $G$ by deleting vertices $v,v_2$ and adding edges $v_1v_3$, $v_4v_5$, $v_6v_7$. Similarly as above, $G'$ is proper with respect to $G$. Also, we can color $v_2$ and $v$ in order, a contradiction. Hence, $v$ is incident to two $5^+$-faces.

(e) Suppose that $v_3$ is also a semi-bad vertex. Since $v_2$ is a semi-bad vertex,  $v_3v_4$ is contained in two $3$-faces by Lemma \ref{lem:5-vertex-n3-0-m3-4}(d). Consequently, each of $v_1v_2$,  $v_2v_3$, $v_3v_4$ is contained in two $3$-faces. If $v$ is incident to at most one $5^+$-face, then at least one of $f_4,f_5$ is a $4$-face. Without loss of generality, assume that $f_5$ is a $4$-face forming $f_4=vv_1xv_5$ for $x\in N(v_1)\cap N(v_5)$. Clearly, we have $d_2(v)\leq 15$ since each of  $v_1v_2$, $v_2v_3$, $v_3v_4$ is contained in two $3$-faces. 
Let $G'$ be the graph obtained from $G$ by deleting vertices $v,v_2$ and adding edges $v_1v_3$, $v_4v_5$, $v_6v_7$. Similarly as above, we get a contradiction. 
\end{proof}

\begin{lemma}\label{lem:support-properties}
Let $v$ be a support vertex. 
\begin{itemize}
\item[$(a)$] If $n_3(v)=1$, then  $v$ is incident to a $5^+$-faces. 
\item[$(b)$] If $n_3(v)=1$ and $n_4(v)=1$ such that there exists a $5^+$-face incident to  both $v$ and its $3$-neighbour, then  $v$ is incident to two $5^+$-faces. 
\item[$(c)$] If $n_3(v)=1$ and $n_4(v)=2$, then  $v$ is incident to three $5^+$-faces. 
\item[$(d)$] If $n_3(v)=2$, then $v$ is incident to three $5^+$-faces. 
\end{itemize}
\end{lemma}
\begin{proof} 
Let $v$ be a support vertex with $3$-faces $f_1=v_1vv_2$ and $f_2=v_2vv_3$. Then $v_2$ is a bad or a semi-bad vertex, which implies that at least one of $v_1v_2$ and $v_2v_3$ is contained in two $3$-faces.

(a).  Let $n_3(v)=1$.  By Lemma \ref{lem:faces-in-3-face}, the $3$-neighbour of $v$ is either $v_4$ or $v_5$. Assume, without loss of generality,  that $v_5$ is a $3$-vertex. By contradiction, assume that all faces incident to $v$ other than $f_1,f_2$ are $4$-faces. We then say that each pair of $\{v_3,v_4\}, \{v_4,v_5\}, \{v_1,v_5\}$ has a common neighbour in $G-v$. Clearly, we have $d_2(v)\leq 15$. Thus, if we set $G'=G-v+\{v_1v_5,v_2v_5,v_4v_5\}$, then $G'$ would be proper with respect to $G$. By minimality, the graph $G'$ has a 2-distance coloring $f$ with $16$ colors. Since $v$ has at most $15$ forbidden colors, we can color $v$ with an available color, a contradiction. Hence, $v$ is incident to a $5^+$-face.

(b). Let $n_3(v)=1$ and $n_4(v)=1$. Suppose that there exists a $5^+$-face incident to  both $v$ and its $3$-neighbour.  By Lemma \ref{lem:faces-in-3-face}, the $3$-neighbour of $v$ is either $v_4$ or $v_5$. Assume, without loss of generality,  that $v_5$ is a $3$-vertex.   Denote by $f_1,f_2,\ldots,f_5$ the faces incident to $v$ in clockwise order on the plane. By our assumption, $f_4$ or $f_5$ is a $5^+$-face.
Assume to the contrary that $v$ is incident to two $4$-faces, i.e., two of $f_3,f_4,f_5$ are $4$-face. We then have $d_2(v)\leq 15$, and deduce that $f_3$ is a $4$-face.
When we set $G'=G-v+\{v_1v_5,v_2v_5,v_4v_5\}$, $G'$ is proper with respect to $G$. Similarly as above, we get a contradiction. Hence, $v$ is incident to two $5^+$-faces.

(c).  Let $n_3(v)=1$ and $n_4(v)=2$. By Lemma \ref{lem:faces-in-3-face}, the $3$-neighbour of $v$ is either $v_4$ or $v_5$. We may suppose that $v_5$ is a $3$-vertex. Assume to the contrary that $v$ is incident to at most two $5^+$-faces. This implies that $v$ is incident to a $4$-face. Since at most one of $v_1,v_3$ is a $4$-vertex by Remark \ref{rem:specil vertices}(b), we may assume that $v_3,v_4$ are $4$-vertices due to $n_4(v)= 2$ and $n_3(v)=1$.  In such a case, we have $d_2(v)\leq 15$ and so we set $G'=G-v+\{v_3v_4,v_4v_5,v_5v_1\}$.  Clearly,  $G'$ is proper with respect to $G$.
Similarly as above, we get a contradiction.

(d). Let $n_3(v)=2$. So, $v_4$ and $v_5$ are $3$-vertices by Lemma \ref{lem:faces-in-3-face}.  Suppose that $v$ is incident to a $4$-face.  In such a case, we have $d_2(v)\leq 15$ and so we set $G'=G-v+\{v_3v_4,v_4v_5,v_5v_1\}$.  Clearly,  $G'$ is proper with respect to $G$.
Similarly as above, we get a contradiction. 
\end{proof}

In the rest of the paper, we will apply discharging to show that $G$ does not exist. We assign to each vertex $v$ a charge $\mu(v)=d(v)-4$ and to each face $f$ a charge $\mu(f)=\ell(f)-4$. By Euler's formula, we have
$$\sum_{v\in V}\left(d(v)-4\right)+\sum_{f\in F}(\ell(f)-4)=-8$$

We next present some rules and redistribute accordingly. Once the discharging finishes, we check the final charge $\mu^*(v)$ and $\mu^*(f)$. If $\mu^*(v)\geq 0$ and $\mu^*(f)\geq 0$, we get a contradiction that no such a counterexample can exist.


\subsection{Discharging Rules} \label{sub:} ~\medskip

We apply the following discharging rules.

\begin{itemize}
\item[\textbf{R1:}] Every $3$-face  receives $\frac{1}{3}$  from each of its incident vertices.
\item[\textbf{R2:}] Every $3$-vertex receives $\frac{1}{9}$  from each of its  $5$-neighbours.
\item[\textbf{R3:}] Every $3$-vertex receives $\frac{1}{3}$  from each of its  incident $5^+$-faces.
\item[\textbf{R4:}] Every $4$-vertex receives $\frac{1}{5}$  from each of its  incident $5^+$-faces.
\item[\textbf{R5:}] Every $5$-vertex receives $\frac{1}{5}$  from each of its  incident $5^+$-face $f$ if $f$ does not contain any $3$-neighbour of $v$.
\item[\textbf{R6:}] Every $5$-vertex receives $\frac{1}{9}$  from each of its  incident $5^+$-face $f$ if $f$ contains a $3$-neighbour of $v$.

\item[\textbf{R7:}] Let $v$ be a $5$-vertex with $m_3(v)\leq 3$, and let $u$ be a $4$-neighbour of $v$.
\begin{itemize}
\item[(a)] If $uv$ belongs to only one $5^+$-face, then $v$ gives $\frac{1}{15}$ to $u$. 
\item[(b)] If $uv$ belongs to two $5^+$-faces, then $v$ gives $\frac{2}{15}$ to $u$. 
\end{itemize}

\item[\textbf{R8:}] Let $v$ be a strong vertex with $3$-faces $f_1=v_1vv_2$, $f_2=v_2vv_3$, and $f_3=v_4vv_5$. Suppose that  $v_2$ is a bad or a semi-bad vertex. 
\begin{itemize}

\item[(a)] If exactly one of $v_1v_2,v_2v_3$ is contained in two $3$-faces, then $v$ gives $\frac{2}{15}$ to $v_2$. 
\item[(b)] If both $v_1v_2$ and $v_2v_3$ are contained in two $3$-faces, then $v$ gives $\frac{1}{5}$ to $v_2$. 

\end{itemize}




\item[\textbf{R9:}] Let $v$ be a good vertex with $3$-faces $f_1=v_1vv_2$, $f_2=v_2vv_3$, and $f_3=v_3vv_4$.
 If  $v_2$ (or $v_3$) is a  semi-bad vertex, then it receives $\frac{1}{5}$ from $v$.

\item[\textbf{R10:}] Every support vertex $v$ gives $\frac{1}{3}$ to each of its bad and semi-bad neighbour $u$ if $vu$ is contained in two $3$-faces.



\end{itemize}

\vspace*{1em}

\noindent
\textbf{Checking} $\mu^*(v), \mu^*(f)\geq 0$, for $v\in V(G), f\in F(G)$\medskip

We initially show that $\mu^*(f)\geq 0$ for each $f\in F(G)$. Given a face $f\in F(G)$, 
if $f$ is a $3$-face, then it receives $\frac{1}{3}$ from each of incident vertex by R1. 
If $f$ is a $4$-face, then $\mu(f)=\mu^*(f)= 0$.
Let $f$ be a $5$-face. Notice that $f$ is not incident to more than  two   $3$-vertices by Corollary \ref{cor:vertices-on-5-face}(a). If $f$ is incident to two $3$-vertices, which  are clearly non-adjacent by Lemma \ref{lem:3-vertex-no-3-4-neighbour}, then all the other vertices incident to $f$ are $5$-vertices by Corollary \ref{cor:vertices-on-5-face}(b), and so each of them receives  $\frac{1}{9}$ from $f$ by R6. Also, $f$ sends $\frac{1}{3}$ to both $3$-vertices incident to $f$ by R3. Thus, $\mu^*(f)= 0$. If $f$ is incident to exactly one $3$-vertex, then
$f$ is incident to at most two $4$-vertices by Corollary \ref{cor:vertices-on-5-face}(c). Thus two $5$-vertices incident to $f$ receive  $\frac{1}{9}$ from $f$ by R6, and the other $4^+$-vertices receive at most $\frac{1}{5}$ from $f$ by R4-R5. Thus $\mu^*(f)= 0$ after $f$ sends $\frac{1}{3}$ to its incident $3$-vertex by R3. Finally, if $f$ is not incident to any $3$-vertex, then  $\mu^*(f)\geq  0$ after $f$ sends at most $\frac{1}{5}$ to its  each incident vertex by R4-R5.

If $f$ is a $6^+$-face, then $\mu(f)=\ell(f)-4 \geq 2$, similarly as above, we have  $\mu^*(f)\geq 0$ by Corollary \ref{cor:number-of-3-vert}.
Consequently, we have $\mu^*(f)\geq 0$ for each $f\in F(G)$.  \medskip

We now pick a vertex $v\in V(G)$ with $d(v)=k$. By Lemma \ref{lem:deg-geq-3}, we have $k\geq 3$. \medskip

\textbf{(1).} Let $k=3$. By Lemma \ref{lem:faces-in-3-face}, $v$ is not incident to any $3$-face. Also,  all neighbours of $v$ are $5$-vertices  by Lemma \ref{lem:3-vertex-no-3-4-neighbour}. So, $v$ receives $\frac{1}{9}$ from each neighbour by R2. On the other hand, $v$ is incident to at least two $5^+$-faces by Lemma \ref{lem:faces-in-3-face}.  It follows from applying R3 that $v$ receives $\frac{1}{3}$  from  each of its incident $5^+$-faces, and so  $\mu^*(v)\geq  d(v)-4+ 3\times \frac{1}{9}+2\times \frac{1}{3}=0$. \medskip

\textbf{(2).} Let $k=4$. Denote by $v_1,v_2,v_3, v_4$ the neighbours of $v$. Recall that all $v_i$'s are $4^+$-vertices by Lemma \ref{lem:3-vertex-no-3-4-neighbour}. Obviously $\mu^*(v) \geq 0$ when $m_3(v)=0$ since $v$ does not have to give a charge to any vertex. We may therefore assume that  $m_3(v)\geq 1$. In addition, we have $m_3(v)\leq 2$ by Lemma \ref{lem:4-vertex-has-no-two-3-face}.

First suppose that $m_3(v)=1$. Let $f=v_1vv_2$ be the $3$-face incident to $v$. Observe that if $v$ is incident to at least two $5^+$-faces, then $v$ receives $\frac{1}{5}$ from each of its incident $5^+$-face by R4, and so $\mu^*(v)\geq 0$ after $v$ sends $\frac{1}{3}$ to $f$ by R1. Thus we may further assume that $v$ is incident to at most one $5^+$-face.
If $f$ is a $(4,4,4^+)$-face, then, by Lemma \ref{lem:4-vertex-bad-3}, $v$ is incident to two $5^+$-faces, it is a contradiction with our assumption. 
Suppose now that $f$ is a $(4,5,5)$-face.  Notice that both $v_3$ and $v_4$ are $5$-vertices by Proposition \ref{prop:4-vertex-4-5-5-no-4}.  Also, by Lemma \ref{lem:4-vertex-semi-bad-3}, $v$ is incident to exactly one $5^+$-face, say $f'$. 
This implies that $m_4(v)=2$. 
Suppose first that $f$ and $f'$ are adjacent, then $v_1v_2$ is not contained in two $3$-faces by Proposition \ref{prop:4-vertex-no-common-m1}. Thus, for every $i\in [4]$, we conclude that $m_3(v_i)\leq 3$. 
Since $f$ and $f'$ are adjacent, we may assume, without loss of generality, that $f'$ contains the path $v_2vv_3$.  So, $v$ receives at least $\frac{1}{15}$ from each $v_2$ and $v_3$ by R7(a), and $\frac{1}{5}$ from $f'$ by R4.  Thus, $v$ receives totally at least $\frac{2}{15}+\frac{1}{5}$ from its $5$-neighbours, and its incident $5^+$-face. Consequently  $\mu^*(v)\geq 0$ after $v$ sends $\frac{1}{3}$ to $f$ by R1. 
Suppose next that $f$ and $f'$ are non-adjacent. It turns out that $f'$ is a $5^+$-face containing the path $v_3vv_4$. Observe that both $v_3$ and $v_4$ are incident to at least two $4^+$-faces, so  we have $m_3(v_3)\leq 3$ and $m_3(v_4)\leq 3$. It then follows that $v$ receives at least $\frac{1}{15}$ from each $v_3$ and $v_4$ by R7(a), and $\frac{1}{5}$ from $f'$ by R4.  Thus, $v$ receives totally at least $\frac{1}{3}$ from its $5$-neighbours, and its incident $5^+$-face. Consequently  $\mu^*(v)\geq 0$ after $v$ sends $\frac{1}{3}$ to $f$ by R1. 

We now suppose that  $m_3(v)=2$. Let $f_1,f_2$ be $3$-faces incident to $v$. By Lemma \ref{lem:4-vertex-has-no-two-3-face}, all neighbours of $v$ are $5$-vertices, and $v$ is incident to two $5^+$-faces, say $f_3,f_4$. Moreover, no edge in  $G[N(v)]$ is contained in two $3$-faces by Proposition \ref{prop:4-vertex-no-common}. Thus, for every $i\in [4]$, we have $m_3(v_i)\leq 3$. By applying R4, $v$ receives $\frac{1}{5}$ from each of its incident $5^+$-face. If $f_1$ and $f_2$ are non-adjacent, then $v$ receives $\frac{1}{15}$ from its each neighbour by R7(a).  If $f_1$ and $f_2$ are adjacent, then there exists $u\in N(v)$ such that $uv$ belongs to two $5^+$-faces. So, $v$   receives $\frac{2}{15}$ from $u$ by R7(b), and $\frac{1}{15}$ from its every other neighbours belonging to $f_3$ or $f_4$ by R7(a). Thus, $v$ receives totally at least $2\times\frac{1}{5}+\frac{2}{15}+2\times\frac{1}{15}$ from its $5$-neighbours, and its incident $5$-faces. Consequently  $\mu^*(v)\geq 0$ after $v$ sends $\frac{1}{3}$ to each $f_1$ and $f_2$ by R1.  \medskip

\textbf{(3).} Let $k=5$. Denote by $v_1,\ldots, v_5$ the neighbours of $v$ in clockwise order, and let $f_1,f_2,\ldots, f_5$ be faces incident to $v$. By Lemma \ref{lem:faces-on-5-vertex}, $n_3(v)\leq 4$ and $m_3(v)\leq 4$. 
We first note that if $v$ is incident to at most one $3$-face, then  $\mu^*(v)\geq 0$ after $v$  sends $\frac{1}{9}$ to each of its $3$-neighbours by R2; $\frac{1}{3}$ to its incident $3$-face by R1 (if exists); at most $\frac{2}{15}$ to each of its $4$-neighbours by R7. We may therefore assume that $m_3(v)\geq 2$. 
Observe that if $v$ is adjacent to at least three $3$-vertices, then at most one of $f_i$'s can be $3$-face, a contradiction with the assumption. Thus we also assume that $n_3(v)\leq 2$. \medskip


\textbf{(3.1).} Let $n_3(v)=2$. Then at most two of $f_i$'s can be $3$-face by Lemma \ref{lem:faces-in-3-face}.  Indeed, by the assumption of $m_3(v)\geq 2$, we say that exactly two of $f_i$'s are $3$-faces. Let $f_1,f_2$ be those $3$-faces. Clearly, $f_1$ and $f_2$ are adjacent, say $f_1=v_1vv_2$, $f_2=v_2vv_3$. Also, $v_4$ and $v_5$ are $3$-vertices.
If $v$ is a support vertex, then $v_2$ is a bad or a semi-bad vertex. In such a case,  
 $v$ is incident to three $5^+$-faces by Lemma \ref{lem:support-properties}(d), and at  most one of $v_1, v_3$ is a $4$-vertex by Proposition \ref{prop:  v has two 4 neighbour}.  By R6, $v$ receives $\frac{1}{9}$ from each of its incident $5^+$-face. Thus $\mu^*(v)\geq 0$ after $v$ sends $\frac{1}{3}$ to each of its incident $3$-faces by R1, at most $\frac{1}{3}$ to $v_2$ by R10, $\frac{1}{9}$ to each of its $3$-neighbours by R2 and $\frac{1}{15}$ to its $4$-neighbour by R7(a) (if exists).
  
Let us now assume that $v$ is not a support vertex.  Since at most one of $v_i$'s is a $4$-vertex by Proposition \ref{prop:  v has two 4 neighbour}, we have  $\mu^*(v)\geq 0$ after $v$  sends  $\frac{1}{3}$ to each of its incident $3$-faces by R1, $\frac{1}{9}$ to each of its $3$-neighbours by R2, and $\frac{1}{15}$ to its $4$-neighbour by R7(a) (if exists). \medskip




\textbf{(3.2).} Let $n_3(v)=1$. Clearly, we have $2\leq m_3(v)\leq 3$ by Lemma \ref{lem:faces-in-3-face}. Denote by $v_5$ the $3$-neighbour of $v$.  \medskip

\textbf{(3.2.1).} Let $ m_3(v)=2$. By Lemma \ref{lem:5-vertex-three-4-neigh}(a), $v$ has at most three $4$-neighbours, i.e., $n_4(v)\leq 3$. 
Notice that $v$ is incident to  a $5^+$-face when $v$ has a $4$-neighbour by Lemma \ref{lem:5-vertex-three-4-neigh}(b). In such a case, $v$ receives at least $\frac{1}{9}$ from its incident $5^+$-face by R5-R6.

Let us first assume that $f_1$ and $f_2$ are non-adjacent, say $f_1=v_1vv_2$ and $f_2=v_3vv_4$. Then, we note that no $vv_i$ is contained in two $5^+$-face for $i \in [5]$, since $d(v_5)=3$. This means that the rule R7(b) cannot be applied for $v$. Thus,  $\mu^*(v) \geq 0$ after $v$ sends $\frac{1}{3}$ to each of its incident $3$-faces by R1, $\frac{1}{9}$ to its $3$-neighbour by R2, and  $\frac{1}{15}$ to each of its $4$-neighbours by R7(a).

Assume further that $f_1$ and $f_2$ are adjacent, say $f_1=v_1vv_2,$ and $f_2=v_2vv_3$.  Denote by $f_1,f_2,\ldots,f_5$ the faces incident to $v$ in clockwise order on the plane.  Notice that if there exists a $4$-vertex $v_i\in N(v)$ such that $vv_i$ is contained in two $5^+$-faces, then $v$ receives totally at least $\frac{1}{9}+\frac{1}{5}$ from its incident $5^+$-faces by R5-R6, since $v_5$ is a $3$-vertex. In fact, only $vv_4$ among all $vv_i$'s can belong to two $5^+$-faces. So, $v$ sends totally at most  $2\times \frac{1}{15}+\frac{2}{15}$ to its all $4$-neighbours by R7 because of $d(v_5)=3$ and $n_4(v)\leq 3$. Notice that $v$ may be a support vertex or not, so we will consider these two cases separately as follows.

Suppose that $v$ is not a support vertex. So $v_2$ does not receive any charge from $v$, even if it is $4$-vertex, according to R7. If $v$ has no $4$-neighbour, then  $\mu^*(v) \geq 0$ after $v$ sends $\frac{1}{3}$ to each of its incident $3$-faces by R1, $\frac{1}{9}$ to its $3$-neighbour by R2. If $v$ has a $4$-neighbour, then  $v$ receives at least $\frac{1}{9}$ from its incident $5^+$-face by R5-R6, which exists by Lemma \ref{lem:5-vertex-three-4-neigh}(b), and so $\mu^*(v) \geq 0$ after $v$ sends $\frac{1}{3}$ to each of its incident $3$-faces by R1, $\frac{1}{9}$ to its $3$-neighbour by R2, $\frac{1}{15}$ to each $v_1,v_3$ by R7(a) when they are $4$-vertices, at most  $\frac{2}{15}$ to $v_4$ by R7(a)-(b) when it is a $4$-vertex.

Suppose now that $v$ is a support vertex. Then $v_2$ is a bad or a semi-bad vertex, and so $v_2$ has two common neighbours with at least one of $v_1,v_3$. By Remark \ref{rem:specil vertices}(a)-(b), $v_2$ has at most one $4$-neighbour. This infer that $v$ has at most two $4$-neighbours. 
In particular, $v$ is incident to a $5^+$-face by Lemma \ref{lem:support-properties}(a). 
We distinguish the rest of the proof according to the number of $4$-neighbours of $v$.

Let $n_4(v)=0$. Then $v$ receives at least $\frac{1}{9}$ from its incident $5^+$-face by R5-R6, and so we have $\mu^*(v)\geq 0$  after $v$ sends $\frac{1}{3}$ to each of its incident $3$-faces by R1, $\frac{1}{3}$ to $v_2$ by R10, $\frac{1}{9}$ to its $3$-neighbour by R2.

Let $n_4(v)=1$.  Recall that $v$ is incident to at least one $5^+$-face. 
First, we suppose that $f_4,f_5$ are $4$-faces, which implies that  $f_3$ is a $5^+$-face. In this case, $v$ receives $\frac{1}{5}$ from $f_3$ by R5.  Thus we have  $\mu^*(v)\geq 0$  after $v$ sends $\frac{1}{3}$ to each of its incident $3$-faces by R1, $\frac{1}{3}$ to $v_2$ by R10, $\frac{1}{9}$ to its $3$-neighbour by R2, and $\frac{1}{15}$ to its $4$-neighbour by R7. We further suppose that  $f_4$ or $f_5$ is a $5^+$-face. It then follows from Lemma \ref{lem:support-properties}(b) that $v$ is incident to two $5^+$-faces. 
If $f_3$ is a $5^+$-face, then $v$ receives $\frac{1}{5}$ from $f_3$ by R5 and  $\frac{1}{9}$ from its other incident $5^+$-face by R6.  Thus we have  $\mu^*(v)\geq 0$  after $v$ sends $\frac{1}{3}$ to each of its incident $3$-faces by R1, $\frac{1}{3}$ to $v_2$ by R10, $\frac{1}{9}$ to its $3$-neighbour by R2, at most $\frac{2}{15}$ to its $4$-neighbour by R7. 
If $f_3$ is not a $5^+$-face, then $v$ receives at least $\frac{1}{9}$ from each of its incident $5^+$-face by R5-R6.  Thus we have  $\mu^*(v)\geq 0$  after $v$ sends $\frac{1}{3}$ to each of its incident $3$-faces by R1, $\frac{1}{3}$ to $v_2$ by R10, $\frac{1}{9}$ to its $3$-neighbour by R2, $\frac{1}{15}$ to its $4$-neighbour by R7(a).

Let $n_4(v)=2$. Recall that $v_2$ has at most one $4$-neighbour. So we may assume that $v_3,v_4$ (or $v_1,v_4$) are $4$-neighbours of $v$. It follows from Lemma \ref{lem:support-properties}(c) that  $v$ is incident to three $5^+$-faces. In this case, by R5-R6, $v$ receives at least $\frac{1}{9}$ from  incident $5^+$-face. Therefore we have  $\mu^*(v)\geq 0$  after $v$ sends $\frac{1}{3}$ to each of its incident $3$-faces by R1, $\frac{1}{3}$ to $v_2$ by R10, $\frac{1}{9}$ to its $3$-neighbour by R2, $\frac{1}{15}$ to $v_3$ by R7(a), and  $\frac{2}{15}$ to $v_4$ by R7(b).\medskip

\textbf{(3.2.2).} Let $ m_3(v)=3$. Denote by $f_1,f_2,\ldots,f_5$ the faces incident to $v$ in clockwise order on the plane.  By Lemma \ref{lem:faces-in-3-face} together with the fact $d(v_5)=3$, we say that the $3$-faces incident to $v$ are forming $f_1=v_1vv_2$, $f_2=v_2vv_3$, and $f_3=v_3vv_4$. Since $d(v_5)=3$, $v$ cannot be a good vertex by Lemma \ref{lem:good-properties}(a). Moreover, by Lemma \ref{lem:5-vertex-three-4-neigh}(c), $v$ has at most one $4$-neighbour, and at least one of $f_4, f_5$ is a $5^+$-face. Suppose first that $v$ has no $4$-neighbour.  By R6, $v$ receives $\frac{1}{9}$ from each of its incident $5^+$-face. Thus,  $\mu^*(v)\geq 0$ after $v$ sends $\frac{1}{3}$ to each of its incident $3$-faces by R1 and $\frac{1}{9}$ to its $3$-neighbour by R2.   
Suppose now that  $v$ has exactly one $4$-neighbour. In this case, $v$ is incident to two $5^+$-faces by Lemma \ref{lem:5-vertex-three-4-neigh}(d), so $f_4, f_5$ are $5^+$-faces. By R6, $v$ receives $\frac{1}{9}$ from each of its incident $5^+$-faces. Thus,  $\mu^*(v)\geq 0$ after $v$ sends $\frac{1}{3}$ to each of its incident $3$-faces by R1, $\frac{1}{9}$ to its $3$-neighbour by R2, and $\frac{1}{15}$ to its $4$-neighbour by R7(a). \medskip

\textbf{(3.3).} Let $n_3(v)=0$. Namely, all neighbours of $v$ are $4^+$-vertices. By Lemma \ref{lem:faces-on-5-vertex}, we have $m_3(v)\leq 4$. If $m_3(v)\leq 1$, then $\mu^*(v)\geq 0$ after $v$ sends $\frac{1}{3}$ to each of its incident $3$-faces by R1 and at most $\frac{2}{15}$ to each of its $4$-neighbours by R7. So we may assume that $2\leq m_3(v)\leq 4$. \medskip

\textbf{(3.3.1).} Let  $m_3(v)= 2$. Denote by $f_1,f_2$ the $3$-faces incident to $v$. Notice that $v$ is incident to three $5^+$-faces when all neighbours of  $v$ are $4$-vertices by Lemma \ref{lem:5-vertex-n3-0}(a). 
First of all, by R5, $v$ receives $\frac{1}{5}$ from  each of its incident $5^+$-faces. If $f_1$ and $f_2$ are non-adjacent, then $\mu^*(v)\geq 0$ after $v$ sends $\frac{1}{3}$ to each of its incident $3$-faces by R1, $\frac{1}{15}$ to each of its $4$-neighbours incident to $f_1$ or $f_2$ by R7(a), $\frac{2}{15}$ to its $4$-neighbour incident to neither $f_1$ nor $f_2$ by R7(b).  
We further assume that $f_1$ and $f_2$ are adjacent, say $f_1=v_1vv_2$ and $f_2=v_2vv_3$. 
If $v$ is not a support vertex, then $\mu^*(v)\geq 0$ after $v$ sends $\frac{1}{3}$ to each of its incident $3$-faces by R1, $\frac{1}{15}$ to each of its $4$-neighbours incident to $f_1$ or $f_2$ by R7(a), $\frac{2}{15}$ to each of its $4$-neighbours incident to neither $f_1$ nor $f_2$ by R7(b).  
We now assume that $v$ is a support vertex. Then $v_2$ is a bad or a semi-bad vertex, which implies that $v_2$ has at most one $4$-neighbour by Remark \ref{rem:specil vertices}. We then deduce that $v$ has at most three $4$-neighbours. Notice that the rule R7 cannot be applied for $v$ when it is not incident to any $5^+$-face. In fact, if $v$ is not incident to any $5^+$-face, then    $\mu^*(v)\geq 0$ after $v$ sends $\frac{1}{3}$ to each of its incident $3$-faces by R1, $\frac{1}{3}$ to $v_2$ by R10.
Therefore we suppose that $v$ is incident to a $5^+$-face.  By R5, $v$ receives $\frac{1}{5}$ from  each of its incident $5^+$-face. If $v$ is incident to only one $5^+$-face, then $\mu^*(v)\geq 0$ after $v$ sends $\frac{1}{3}$ to each of its incident $3$-faces by R1, $\frac{1}{3}$ to $v_2$ by R10, and $\frac{1}{15}$ to each of its $4$-neighbours by R7(a). If $v$ is incident to two $5^+$-faces, then $v$ receives totally at least $2\times \frac{1}{5}$ from its incident $5$-faces, and so $\mu^*(v)\geq 0$ after $v$ sends $\frac{1}{3}$ to each of its incident $3$-faces by R1, $\frac{1}{3}$ to $v_2$ by R10, and at most $\frac{2}{15}$ to each of its $4$-neighbours by R7. \medskip

\textbf{(3.3.2).} Let $m_3(v)=3$. Denote by $f_1,f_2,f_3$ the $3$-faces incident to $v$.
Notice that $v$ can have at most three $4$-neighbours by Lemma \ref{lem:5-vertex-n3-0}(b). 

Suppose first that $v$ is a strong vertex. Then we may assume that $f_1=v_1vv_2$, $f_2=v_2vv_3$, and $f_3=v_4vv_5$.  Recall that $v$ is incident to a $5^+$-face, and $v_2$ is a bad or semi-bad vertex. By R5, the vertex $v$ receives $\frac{1}{5}$ from each of its incident $5^+$-face.
If $v$ has two $4$-neighbours, then $v$ would have two $5^+$-faces by Lemma \ref{lem:strong-properties}(a). 
In this case, $v$ receives totally $2\times \frac{1}{5}$ from its incident $5^+$-faces. Thus, $\mu^*(v)\geq 0$ after $v$ sends $\frac{1}{3}$ to each of its incident $3$-faces by R1, at most $\frac{1}{5}$ to $v_2$ by R8, and $\frac{1}{15}$ to each of its $4$-neighbours by R7(a) (if exists). 
Assume further that $v$ has at most one $4$-neighbour. 
If $v_2$ has two common neighbours with only one of $v_1,v_3$, then $\mu^*(v)\geq 0$ after $v$ sends $\frac{1}{3}$ to each of its incident $3$-faces by R1, $\frac{2}{15}$ to $v_2$ by R8(a), and $\frac{1}{15}$ to its $4$-neighbour by R7(a) (if exists). 
If $v_2$ has two common neighbours with each of $v_1,v_3$, then either $v$ has no $4$-neighbour or $v$ is incident to  two $5^+$-face by Lemma \ref{lem:strong-properties}(b). For each case, we have  $\mu^*(v)\geq 0$ after $v$ sends $\frac{1}{3}$ to each of its incident $3$-faces by R1, $\frac{1}{5}$ to $v_2$ by R8(b), and $\frac{1}{15}$ to each of its $4$-neighbours by R7(a) (if exists).  


Next we suppose that $v$ is a good vertex. Then we have $f_1=v_1vv_2$, $f_2=v_2vv_3$, and $f_3=v_3vv_4$ such that $v_2$ is a semi-bad vertex, and each of $v_1v_2,v_2v_3$ is contained in two $3$-faces. Also, by Lemma \ref{lem:good-properties}(b)-(c), $v$ has at most one $4$-neighbour, and $v$ is incident to a $5^+$-face. 
If $v$ has a $4$-neighbour, then $f_4$ and	$f_5$ are $5^+$-faces by Lemma \ref{lem:good-properties}(d). It follows from Proposition \ref{prop:good-no-two-semi-bad} that only one of $v_2,v_3$ can be semi-bad vertex. By R5, the vertex $v$ receives $\frac{1}{5}$ from each $f_4,f_5$.  Thus $\mu^*(v)\geq 0$ after $v$ sends $\frac{1}{3}$ to its each incident $3$-face by R1, $\frac{1}{5}$ to $v_2$ by R9, at most $\frac{2}{15}$ to its $4$-neighbour by R7. 
Suppose next that $v$ has no $4$-neighbour. If $v_3$ is not a semi-bad vertex, 
then  $v$ receives $\frac{1}{5}$ from its incident $5^+$-face by R5, and sends   $\frac{1}{5}$ to $v_2$ by R9. Consequently,  $\mu^*(v)\geq 0$ after $v$ sends $\frac{1}{3}$ to its each incident $3$-face by R1.
If $v_3$ is a  semi-bad vertex, 
then  $v$ is incident to two $5^+$-faces by Lemma \ref{lem:good-properties}(e).  It follows from applying R5 that $v$ receives $\frac{1}{5}$ from each of its incident $5^+$-face, and sends   $\frac{1}{5}$ to each $v_2,v_3$ by R9. Consequently,  $\mu^*(v)\geq 0$ after $v$ sends $\frac{1}{3}$ to its each incident $3$-face by R1. \medskip

Finally, suppose that $v$ is neither a strong nor a good vertex.  Notice that the rule R7 cannot be applied for $v$ when it is not incident to any $5^+$-face. So, if none of $f_4,f_5$ is a $5^+$-face, then $\mu^*(v)\geq 0$ after $v$ sends $\frac{1}{3}$ to each of its incident $3$-faces by R1.
If only one of $f_4,f_5$ is a $5^+$-face, then  $v$ receives $\frac{1}{5}$ from its incident $5^+$-face by R5, and sends $\frac{1}{15}$ to each of its $4$-neighbours  by R7(a) (if exists). If both $f_4$ and $f_5$ are $5^+$-faces, then  $v$ receives $\frac{1}{5}$ from each of its incident $5$-faces by R5, and sends at most $\frac{2}{15}$ to each of its $4$-neighbours by R7 (if exists). Consequently,  $\mu^*(v)\geq 0$ after $v$ sends $\frac{1}{3}$ to each of its incident $3$-faces by R1. \medskip

\textbf{(3.3.3).} Let $m_3(v)=4$. Let $f_1,f_2,f_3,f_4$ be $3$-faces such that $f_i=v_ivv_{i+1}$ for $i\in [4]$. 
We first observe that if $v$ is not incident to any $5^+$-face, then $v$ cannot have any  $4$-neighbour by Lemma \ref{lem:5-vertex-n3-0-m3-4}(b). Moreover, $v$ cannot have more than one $4$-neighbour by Lemma \ref{lem:5-vertex-n3-0-m3-4}(a). The vertex $v$ is a bad or a semi-bad vertex, we distinguish two cases accordingly.\medskip

\textbf{(3.3.3.1).} Let $v$ be a bad vertex, i.e., $f_5$ is a $4$-face. By Remark \ref{rem:specil vertices}(a), all neighbours of $v$ are  $5$-vertices, also there is no edge $v_iv_{i+1}$ for $i\in [4]$ contained in two $3$-faces.
If one of $v_2,v_3,v_4$ is incident to only two $3$-faces, then it must be a support vertex, and so it gives $\frac{1}{3}$ to $v$ by R10. It follows that  $\mu^*(v)\geq 0$ after $v$ sends $\frac{1}{3}$ to each of its incident $3$-faces by R1. Therefore, we may assume that each of $v_2,v_3,v_4$ is incident to exactly three $3$-faces, since no edge $v_iv_{i+1}$ for $i\in [4]$ is contained in two $3$-faces. By Lemma \ref{lem:deger}(b), $v_3$ is a strong vertex. This particularly implies that $v_3$ is incident to a $5^+$-face $f_j$. 
We assume, without loss of generality, that  $v_2v_3$ is an edge  of $f_j$. We then conclude that both $v_2$ and $v_3$ are strong vertices. It follows that $v$ receives  $\frac{1}{5}$ from $v_3$ by R8(b), and $\frac{2}{15}$ from $v_2$ by R8(a).  Consequently,  $\mu^*(v)\geq 0$ after $v$ sends $\frac{1}{3}$ to each of its incident $3$-faces by R1. \medskip

\textbf{(3.3.3.2).}  Let $v$ be a semi-bad vertex, i.e., $f_5$ is a $5^+$-face. 
By R5, $v$ receives $\frac{1}{5}$ from $f_5$. 
In particular, $v_2,v_3,v_4$ are $5$-vertices by Lemma \ref{lem:deger}(a).

Suppose first that there is no edge $v_iv_{i+1}$ for $i\in [4]$ contained in two $3$-faces.  If one of $v_2,v_3,v_4$ is incident to only two $3$-faces, then it must be support vertex, and so it gives $\frac{1}{3}$ to $v$ by R10. Thus we have  $\mu^*(v)\geq 0$ after $v$ sends $\frac{1}{3}$ to each of its incident $3$-faces by R1. We therefore assume that each of $v_2,v_3,v_4$ is incident to exactly three $3$-faces, since no edge $v_iv_{i+1}$ for $i\in [4]$ is contained in two $3$-faces.
By Lemma \ref{lem:deger}(c), $v_3$ is a strong or a good vertex. However, the latter is not possible since no edge $v_iv_{i+1}$ for $i\in [4]$ is contained in two $3$-faces. So $v_3$ is a strong vertex, and it gives  $\frac{1}{5}$ to $v$ by R8(b). Recall that $v$ has already received $\frac{1}{5}$ from $f_5$. Consequently,  $\mu^*(v)\geq 0$ after $v$ sends $\frac{1}{3}$ to each of its incident $3$-faces by R1. 


We now suppose that there is an edge $v_iv_{i+1}$ for $i\in [4]$ contained in two $3$-faces. The uniqueness of $v_iv_{i+1}$ follows from Lemma \ref{lem:5-vertex-n3-0-m3-4}(d). 
First consider the case that $v_1v_2$ is contained in two $3$-faces. If  $v_3$ is incident to only two $3$-faces, then it must be support vertex by Lemma \ref{lem:deger}(c), and so it gives $\frac{1}{3}$ to $v$ by R10. Thus we have  $\mu^*(v)\geq 0$ after $v$ sends $\frac{1}{3}$ to each of its incident $3$-faces by R1. So we may assume that  $v_3$ is incident to exactly three $3$-faces by Lemma \ref{lem:5-vertex-n3-0-m3-4}(d). 
It follows from Lemma \ref{lem:deger}(c) that $v_3$ is a strong or a good vertex. The latter is not possible since $v_1v_2$ is contained in two $3$-faces, and only one of $v_iv_{i+1}$'s is contained in two $3$-faces Lemma \ref{lem:5-vertex-n3-0-m3-4}(d). Thus, $v_3$ is a strong vertex, and $v$ receives  $\frac{1}{5}$ from $v_3$ by R8. Consequently,  $\mu^*(v)\geq 0$ after $v$ sends $\frac{1}{3}$ to each of its incident $3$-faces by R1.
Let us now assume that $v_2v_3$ is contained in two $3$-faces. In this case, $v_3$ is a good vertex Lemma \ref{lem:deger}(c). It follows from applying R9 that $v$ receives $\frac{1}{5}$ from $v_3$. Consequently,  $\mu^*(v)\geq 0$ after $v$ sends $\frac{1}{3}$ to its each incident $3$-face by R1. \medskip

\section*{Acknowledgement}

The author extends thanks to Kengo Aoki, who pointed out some inconsistencies in an earlier version of this paper.







\begin{thebibliography}{99}

\bibitem{aoki}
Aoki, K. \textit{Improved 2-distance coloring of planar graphs with maximum degree 5}. arXiv preprint arXiv:2307.16394, 2023.


\bibitem{hou}
Hou, J., Jin, Y., Miao, L., and Zhao, Q. \textit{Coloring squares of planar graphs with maximum degree at most five}. Graphs and Combinatorics, 39(2), 20, 2023.


\bibitem{chen}
Chen, M., Miao, L., and Zhou, S.  \textit{2-distance coloring of planar graphs with maximum degree 5}. Discrete Mathematics, 345(4), 112766, 2022.



\bibitem{daniel}
Cranston, D. W. \textit{Coloring, List Coloring, and Painting Squares of Graphs (and other related problems)},   The Electronic Journal of Combinatorics, \textbf{30}, 2, DS25, (2022).






\bibitem{hartke}
Hartke, S. G.,    Jahanbekam, S. and  Thomas, B.,  \textit{The chromatic number of the square of subcubic planar graphs}. arXiv preprint arXiv:1604.06504, (2016). 

\bibitem{van-den}
Heuvel, J. van den,  McGuinness,  S., \textit{Coloring of the square of planar graph},  Journal of Graph Theory, 42  110–124, (2003).



\bibitem{molloy}
Molloy, M. and Salavatipour, M. R., \textit{ A bound on the chromatic number of the square
of a planar graph},  Journal of Combinatorial Theory Series B, 94,  189–213, (2005).

\bibitem{thomassen}
Thomassen, C., \textit{The square of a planar cubic graph is 7-colorable.} Journal of Combinatorial Theory Series B, 128:192–218, (2018).


\bibitem{wegner}
Wegner, G., \textit{Graphs with given diameter and a coloring problem}. Technical report, University of Dormund,
(1977).


\bibitem{west}
West, D. B. Introduction to graph theory, volume 2. Prentice hall Upper Saddle River, 2001.


\bibitem{zhu}
Zhu B. \textit{$L(1, 1)$-Labelling of graphs}, Doctor Thesis, Zhejiang Normal University, 2012.


\bibitem{zhu2}
Zhu, B., Bu, Y.  \textit{Minimum 2-distance coloring of planar graphs and channel assignment}, J. Comb. Optim. 36  55–64, 2018.



\bibitem{zou}
Zou, J., Han, M., and Lai, H. J. \textit{Square Coloring of Planar Graphs with Maximum Degree at Most Five}. arXiv preprint arXiv:2308.01824, 2023.













\end{thebibliography}
\end{document}